\documentclass[10pt]{article}

\usepackage{tikz, float}
\usepackage{amssymb}
\usepackage{url}
\usepackage{amscd}
\usepackage{amsthm}
\usepackage{amsmath}
\usepackage{graphicx}

\newtheorem {Theorem}  {Theorem}

\newtheorem {Problem} {Problem}

\begin{document}
\baselineskip = 15pt
\bibliographystyle{plain}

\title{Undecidability of Translational Tiling of the $3$-dimensional Space\\ with a Set of $6$ Polycubes}
\date{}
\author{Chao Yang\\ 
              School of Mathematics and Statistics\\
              Guangdong University of Foreign Studies, Guangzhou, 510006, China\\
              sokoban2007@163.com, yangchao@gdufs.edu.cn\\
              \\
        Zhujun Zhang\\
              Government Data Management Center of\\
              Fengxian District, Shanghai, 201499, China\\
              zhangzhujun1988@163.com
              }
\maketitle

\begin{abstract}
This paper focuses on the undecidability of translational tiling of $n$-dimensional space $\mathbb{Z}^n$ with a set of $k$ tiles. It is known that tiling $\mathbb{Z}^2$ with translated copies with a set of $8$ tiles is undecidable. Greenfeld and Tao gave strong evidence in a series of works that for sufficiently large dimension $n$, the translational tiling problem for $\mathbb{Z}^n$ might be undecidable for just one tile. This paper shows the undecidability of translational tiling of $\mathbb{Z}^3$ with a set of $6$ tiles.
\end{abstract}


\section{Introduction} 

The phenomenon of aperiodicity and undecidability in tiling was discovered after Hao Wang introduced the plane tiling problem with Wang tiles \cite{wang61}. A \textit{Wang tile} is a unit square with each edge assigned a color. Given a finite set of Wang tiles (see Figure \ref{fig_wang_set} for an example), Wang considered the problem of tiling the entire plane with translated copies from the set, under the conditions that the tiles must be edge-to-edge and the color of common edges of any two adjacent Wang tiles must be the same. This is known as \textit{Wang's domino problem}.

\begin{figure}[H]
\begin{center}
\begin{tikzpicture}

\draw [fill=green!80] (0,0)--(1,1)--(0,2)--(0,0);
\draw [fill=red!80] (0,0)--(2,0)--(0,2)--(2,2)--(0,0);
\draw [fill=yellow!80] (1,1)--(2,2)--(2,0)--(1,1);

\draw [fill=yellow!80] (3+0,0)--(3+1,1)--(3+0,2)--(3+0,0);
\draw [fill=blue!80] (3+0,0)--(3+2,0)--(3+0,2)--(3+2,2)--(3+0,0);
\draw [fill=red!80] (3+1,1)--(3+2,2)--(3+2,0)--(3+1,1);

\draw [fill=red!80] (6+0,0)--(6+1,1)--(6+0,2)--(6+0,0);
\draw [fill=yellow!80] (6+0,0)--(6+2,0)--(6+0,2)--(6+2,2)--(6+0,0);
\draw [fill=green!80] (6+1,1)--(6+2,2)--(6+2,0)--(6+1,1);

\end{tikzpicture}
\end{center}
\caption{A set of $3$ Wang tiles.}\label{fig_wang_set}
\end{figure}
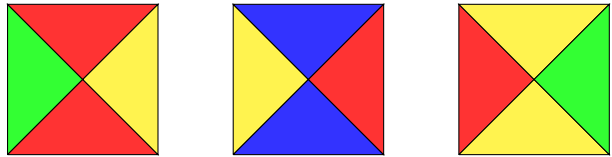

Wang's domino problem was shown to be undecidable by Berger \cite{b66}, and is one of the earliest known undecidable problems. As a crucial part of the proof of undecidability, Berger found a set of Wang tiles that tiles the plane but only tiles the plane non-periodically. Such a set is called \textit{aperiodic}. By combining the facts of the existence of an aperiodic set of Wang tiles and the ability to simulate the Turing machine with Wang tiles, Berger managed to show that Wang's domino problem is undecidable.

\begin{Theorem}[\cite{b66}]\label{thm_berger}
    Wang's domino problem is undecidable.
\end{Theorem}

Since Berger's proof of the undecidability of Wang's domino problem, researchers have found aperiodic sets of Wang tiles with fewer and fewer tiles in the following decades \cite{c96,kari96,r71}. Finally, with the help of computers, Jeandel and Rao proved that $11$ is the smallest possible size of an aperiodic set of Wang tiles \cite{jr21}. Complete characterization of plane tiling by Wang tiles with two colors is obtained \cite{hl11}. Aperiodic sets of even smaller sizes can be found in the general settings of plane tiling (e.g. rotation and reflection of the tiles are allowed in addition to translation). Notably, Penrose first found an aperiodic set of two tiles \cite{p79}. More strikingly, a family of aperiodic monotiles has been discovered by Smith, Myers, Kaplan, and Goodman-Strauss \cite{smith23a,smith23b}, which solves the long-standing Einstein problem. For more aperiodic tiling sets and their applications in quasicrystal, we refer to the books \cite{bg13, gs16}.

\begin{figure}[h]
\begin{center}
\begin{tikzpicture}[scale=0.6]

\draw [ fill=red!80] (0,0)--(1,1)--(0,2)--(0,0);
\draw [ fill=yellow!80] (0,0)--(1,1)--(2,0)--(0,0);
\draw [ fill=blue!80] (2,0)--(1,1)--(2,2)--(2,0);
\draw [ fill=green!80] (0,2)--(1,1)--(2,2)--(0,2);
 
\draw [ fill=blue!80] (3,0)--(4,1)--(3,2)--(3,0);
\draw [ fill=green!80] (3,0)--(4,1)--(5,0)--(3,0);
\draw [ fill=orange!80] (5,0)--(4,1)--(5,2)--(5,0);
\draw [ fill=yellow!80] (3,2)--(4,1)--(5,2)--(3,2);

\draw [thick](8,0)--(11,0)--(11,-1)--(11.5,-1)--(11.5,0)--(12,0)--(12,1.5)--(13,1.5)--(13,3)--(12.5,3)--(12.5,2)--(12,2)--(12,4)--(9.5,4)--(9.5,3.5)--(8.5,3.5)--(8.5,4)--(8,4)--(8,1.5)--(9,1.5)--(9,0.5)--(8.5,0.5)--(8.5,1)--(8,1)--(8,0);

\draw [thick](14,0)--(14.5,0)--(14.5,-0.5)--(15.5,-0.5)--(15.5,0)--(18,0)--(18,0.5)--(18.5,0.5)--(18.5,1)--(18,1)--(18,4)--(17.5,4)--(17.5,3)--(17,3)--(17,4)--(14,4)--(14,2)--(14.5,2)--(14.5,3)--(15,3)--(15,1.5)--(14,1.5)--(14,0);

\end{tikzpicture}
\end{center}
\caption{Wang tiles simulated by polyominoes.} \label{fig_golomb}
\end{figure}
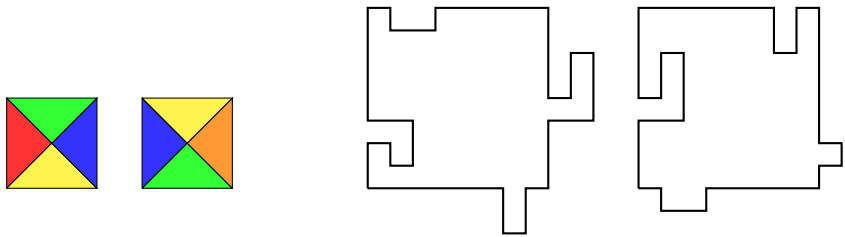

In another trend, Ollinger initiated the study of the undecidability of translational tiling of the plane with a set of a fixed number of tiles by introducing the following $k$-polyomino tiling problem.

\begin{Problem}[$k$-Polyomino Tiling Problem] 
For a fixed positive integer $k$, is there an algorithm to decide whether a set of $k$ polyominoes can tile the plane by translated copies from the set?    
\end{Problem}

It has been shown by Golomb \cite{g70} that Wang's domino problem can be reduced to the problem of tiling with a set of polyominoes. Golomb's reduction method can be illustrated in Figure \ref{fig_golomb}, where a set of Wang tiles is emulated by a set of polyominoes with the same number. The color of Wang tiles can be simulated by bumps and dents added to each side of a large square polyomino. Therefore, the problem of tiling with a set of polyominoes is undecidable in general (i.e. the number of polyominoes is not fixed). Ollinger proved that it remains undecidable even if the number of polyominoes is fixed. 

\begin{Theorem}[\cite{o09}]
    The $11$-polyomino tiling problem is undecidable. 
\end{Theorem}

Based on the existence of an aperiodic set of $8$ polyominoes \cite{ags92}, Ollinger conjectured that the $8$-polyomino tiling problem is undecidable\cite{o09}. Ollinger's conjecture has been confirmed in a series of recent works by Yang and Zhang \cite{yang23,yang23b,yz24}.  

\begin{Theorem}[\cite{yz24}]
   The $8$-polyomino tiling problem is undecidable. 
\end{Theorem}

On the other hand, for $k=1$, it is known that if a single polyomino tiles the plane, then it can tile the plane periodically \cite{bn91}. So the $1$-polyomnino tiling problem is decidable, and there is a fast algorithm to decide whether a polyomino tiles the plane \cite{w15}. The decidability of $k$-polyomino tiling problem is open for $2\leq k\leq 7$. Beauquier and Nivat conjecture that $2$-polyomino tiling problem is decidable \cite{bn91}.

The study of aperiodicity and undecidability of translational tiling of the plane can be naturally extended to spaces of dimensions other than $2$. The problem can be formulated in integer lattice $\mathbb{Z}^n$. A finite subset $T$ of $\mathbb{Z}^n$ is called a \textit{tile}. 

\begin{Problem}[Translational tiling of $\mathbb{Z}^n$ with a set of $k$ tiles] \label{pro_main}
Given a set $S$ of $k$ tiles in $\mathbb{Z}^n$, is there an algorithm to decide whether $\mathbb{Z}^n$ can be tiled by translated copies of $S$?
\end{Problem}

Problem \ref{pro_main} can be equivalently and more geometrically stated as a tiling problem in $\mathbb{R}^n$, where a tile is the union of a finite set of unit hypercubes of the form $\Pi_{i=1}^n [z_i,z_i+1] $ ($z_i\in \mathbb{Z}$). Note that a tile may not be connected in the Euclidean space $\mathbb{R}^n$.

For dimension one ($n=1$), it is shown that for any $k$, if a set of $k$ tiles can tile $\mathbb{Z}$, then it can always tile $\mathbb{Z}$ periodically \cite{s93}, even though the periods can be exponentially long \cite{kolo03, st05}. As a consequence, the problem is decidable in dimension one. For $n=2$ and $k=1$, the periodicity and decidability are shown in \cite{b20}.

The cases for general dimensions have been extensively studied by Greenfeld and Tao \cite{gt21,gt23,gt24a,gt24b}. Greenfeld and Tao show that the translational tiling problem for subsets of $\mathbb{Z}^n$ is undecidable for just one tile, where the dimension $n$ is part of the input of the instances \cite{gt24b}. They also show the existence of aperiodic monotiles for tiling the entire $\mathbb{Z}^n$ \cite{gt24a}. These are strong evidence that the translational tiling of $\mathbb{Z}^n$ with one tile might be undecidable for some sufficiently large fixed  $n$.

The above results can be summarized in Figure \ref{fig_nk} in the $(k,n)$-plane of the parameters, where $k$ is the number of tiles and $n$ is the dimension of the space. The green area is known to be decidable, and the red area is undecidable. Note that the frontier of the area of undecidable is not clearly known yet, especially for dimensions $4$ and above. The dashed line of the boundary of the undecidable area in Figure \ref{fig_nk} is to illustrate the idea that as the dimension goes up, it may require fewer tiles to get undecidable results for translational tiling of $\mathbb{Z}^n$.

\begin{figure}[H]
\begin{center}
\begin{tikzpicture}

\draw [fill=green!20,dashed] (11.9,0)--(0,0)--(0,2)--(1,2)--(1,1)--(11.9,1);

\begin{scope}
    \clip (0,9) rectangle (2,10.9);
\draw [fill=red!20,dashed] (0,11)--(0,9)--(3,9)--(3,11); 
\end{scope}

\begin{scope}
    \clip (5,1) rectangle (11.9,3);
\draw [fill=red!20,dashed] (12,1)--(7,1)--(7,2)--(5,2)--(5,4)--(12,4); 
\end{scope}

\begin{scope}
    \clip (2,3) rectangle (11.9,10.9);

\draw [dashed,fill=red!20] plot [smooth] coordinates {(0,11) (1,9) (2,9) (4,8.5) (4.5,6) (5,3) (5,2) (5.5,1) (12,0) (15,7) (15,15) (5,12)};
\end{scope}

\draw[help lines, color=gray, dashed]  (-0.1,-0.1) grid (11.9,10.9);
\draw[->,ultra thick] (-0.5,0)--(12,0) node[right]{$k$};
\draw[->,ultra thick] (0,-0.5)--(0,11) node[above]{$n$};

\foreach \y in {1,2,3,4}
{
\node at (-0.5,\y-0.5) {\y};
}
\node at (-0.9,9.7) {some};
\node at (-0.8,9.3) {large $n$?};

\foreach \x in {1,...,11}
{
\node at (\x-0.5,-0.5) {\x};
}

\end{tikzpicture}
\end{center}
\caption{Translational tiling problem of $\mathbb{Z}^n$ with a set of $k$ tiles.}\label{fig_nk}
\end{figure}
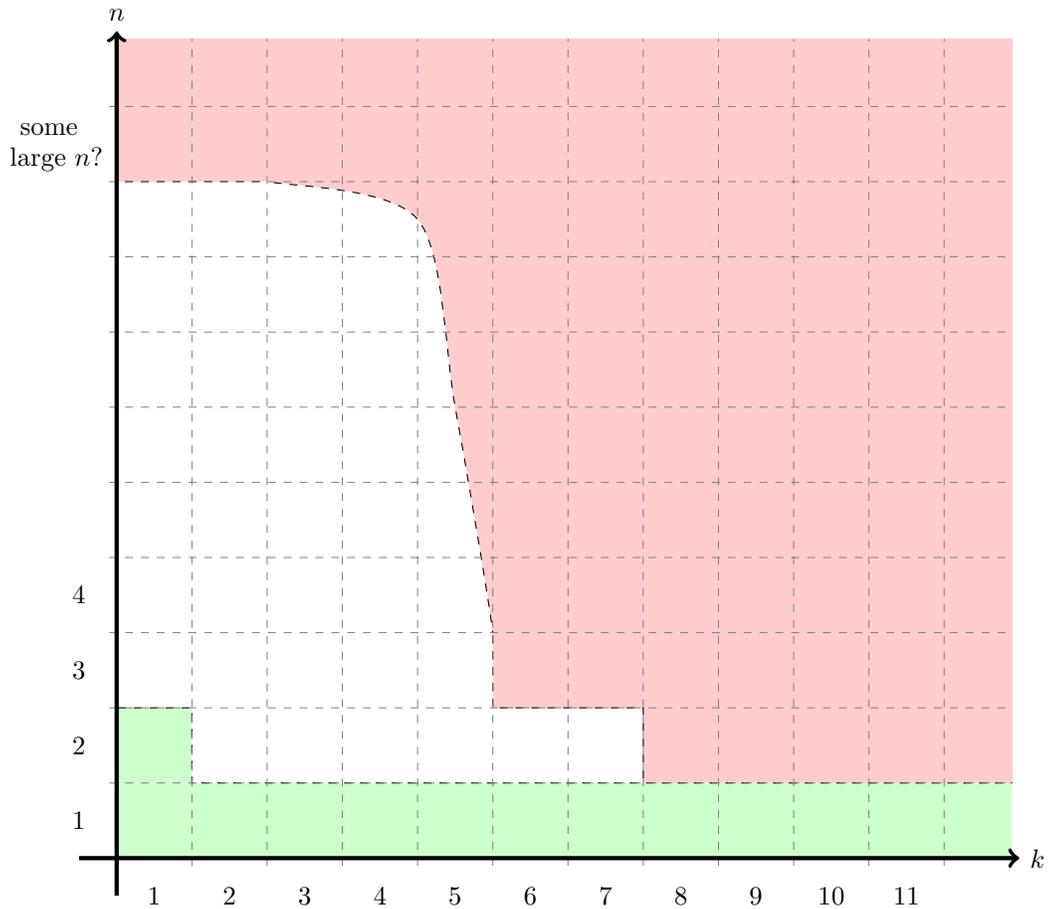

The main contribution of this paper is to enlarge our knowledge of undecidability by showing that the translational tiling of $\mathbb{Z}^n$ is undecidable for $n=3$ and $k=6$, even if all the tiles are connected. This can be stated equivalently in the following theorem.

\begin{Theorem}\label{thm_main}
Translational tiling of $3$-dimensional space with a set of $6$ polycubes is undecidable. 
\end{Theorem}

Theorem \ref{thm_main} will be proved by reduction from Wang's domino problem. We follow the general reduction framework introduced by Ollinger \cite{o09} and subsequently improved by Yang and Zhang \cite{yang23,yang23b,yz24} in studying the $k$-polyomino tiling problems for the plane. To deal with the translational tiling problem, Ollinger's framework requires several rotational duplicate copies of otherwise the same tile. We introduce novel techniques that make use of the $3$-dimensional space to get rid of rotational copies of the same tile in proving Theorem \ref{thm_main}.

The rest of the paper is organized as follows. Section \ref{sec_block} introduces several special polycubes as the building blocks in proving the main result. Section \ref{sec_proof} gives the proof of Theorem \ref{thm_main}. Section \ref{sec_conclude} concludes with a few remarks on future work.

\section{Building Blocks}\label{sec_block}
A \textit{polycube} is a connected union of a finite number of unit cubes putting together face to face. Polycubes are the $3$-dimensional counterpart of polyominoes. A polycube can be depicted by a \textit{layer diagram} in which each horizontal layer of unit cubes is shown as $2$-dimensional layout, from the bottom layer to the top layer (see Figure \ref{fig_polycube} for an example). Layer diagrams are used to describe polycubes which will serve as building blocks of larger polycubes.

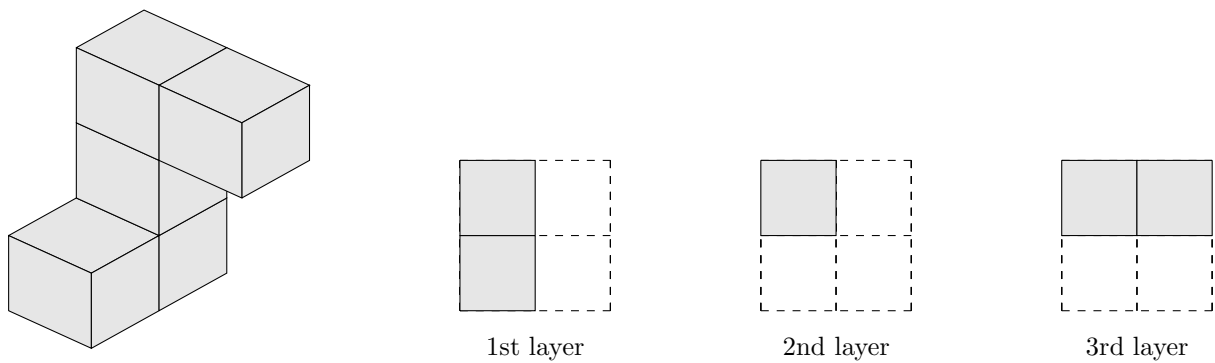
\begin{figure}[H]
\begin{center}
\begin{tikzpicture}[scale=0.5]

\foreach \x in {1.8}
\foreach \y in {7}
{
\draw [ fill=gray!20] (\x,\y)--(\x+2.2,\y-1)--(\x+4,\y)--(\x+1.8,\y+1)--(\x,\y);
}

\foreach \x in {0}
\foreach \y in {2}
{
\draw [ fill=gray!20] (\x,\y)--(\x+2.2,\y-1)--(\x+4,\y)--(\x+1.8,\y+1)--(\x,\y);
}

\foreach \x in {4}
\foreach \y in {6}
{
\draw [ fill=gray!20] (\x,\y)--(\x+2.2,\y-1)--(\x+4,\y)--(\x+1.8,\y+1)--(\x,\y);
}

\foreach \x in {0}
\foreach \y in {0}
{
\draw [ fill=gray!20] (\x,\y)--(\x+2.2,\y-1)--(\x+2.2,\y+1)--(\x,\y+2)--(\x,\y);
}

\foreach \x in {1.8}
\foreach \y in {3,5}
{ 
\draw [ fill=gray!20] (\x,\y)--(\x+2.2,\y-1)--(\x+2.2,\y+1)--(\x,\y+2)--(\x,\y);
}

\foreach \x in {2}
\foreach \y in {-1}
{ 
\draw [ fill=gray!20] (\x+0.2,\y)--(\x+2,\y+1)
--(\x+2,\y+3)--(\x+0.2,\y+2)--(\x+0.2,\y);
}

\foreach \x in {4}
\foreach \y in {0,2}
{ 
\draw [ fill=gray!20] (\x,\y)--(\x+1.8,\y+1)
--(\x+1.8,\y+3)--(\x,\y+2)--(\x,\y);
}
\foreach \x in {6.2}
\foreach \y in {3}
{ 
\draw [ fill=gray!20] (\x,\y)--(\x+1.8,\y+1)
--(\x+1.8,\y+3)--(\x,\y+2)--(\x,\y);
}

\foreach \x in {4}
\foreach \y in {4}
{
\draw [ fill=gray!20] (\x,\y)--(\x+2.2,\y-1)--(\x+2.2,\y+1)--(\x,\y+2)--(\x,\y);
}

\foreach \x in {12,20,28}
\foreach \y in {0,2,4} 
{ 
\draw [dashed] (\x+0,0+\y)--(\x+4,0+\y);
\draw [dashed] (\x,0)--(\x,4);
\draw [dashed] (\x+4,0)--(\x+4,4);
\draw [dashed] (\x+2,0)--(\x+2,4);
}

\draw [fill=gray!20] (12,0)--(14,0)--(14,2)--(12,2)--(12,0);
\draw [fill=gray!20] (12,4)--(14,4)--(14,2)--(12,2)--(12,4);

\draw [fill=gray!20] (20,4)--(22,4)--(22,2)--(20,2)--(20,4);

\draw [fill=gray!20] (28,4)--(30,4)--(30,2)--(28,2)--(28,4);
\draw [fill=gray!20] (32,4)--(30,4)--(30,2)--(32,2)--(32,4);

\node at (14,-1) {$1$st layer};  \node at (22,-1) {$2$nd layer};  \node at (30,-1) {$3$rd layer}; 

\end{tikzpicture}
\end{center}
\caption{A polycube and its layer-by-layer diagram.}\label{fig_polycube}
\end{figure}


A \textit{functional cube} is a $8\times 8 \times 8$ polycube. A \textit{half functional cube} is a $8\times 8\times 4$ polycube with two horizontal sides of length $8$ and a height of $4$. We will construct a set of $6$ polycubes to simulate any given set of Wang tiles in the next section. The basic building blocks of our set of $6$ polycubes are functional cubes and half functional cubes with bumps or dents.


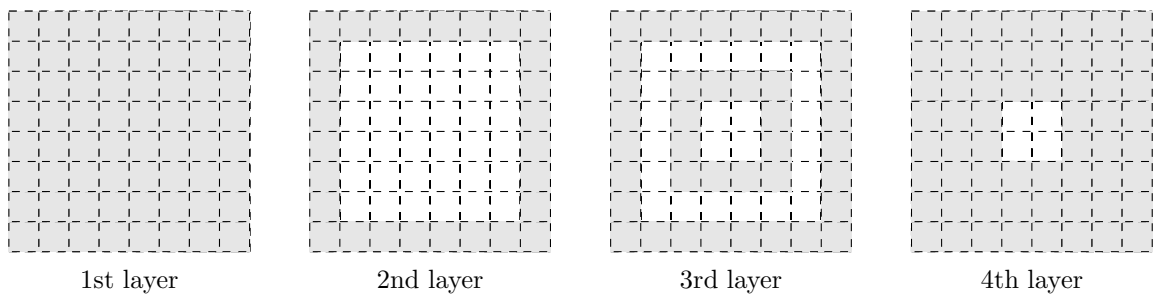
\begin{figure}[H]
\begin{center}
\begin{tikzpicture}[scale=0.4]

\foreach \x in {0,10,20,30}
\foreach \y in {0,...,8} 
{ 
\draw [dashed] (\x+0,0+\y)--(\x+8,0+\y);
\draw [dashed] (\x,0)--(\x,8);
\draw [dashed] (\x+8,0)--(\x+8,8);
\draw [dashed] (\x+7,0)--(\x+7,8);
\draw [dashed] (\x+6,0)--(\x+6,8);
\draw [dashed] (\x+5,0)--(\x+5,8);
\draw [dashed] (\x+4,0)--(\x+4,8);
\draw [dashed] (\x+3,0)--(\x+3,8);
\draw [dashed] (\x+2,0)--(\x+2,8);
\draw [dashed] (\x+1,0)--(\x+1,8);
}

\foreach \x in {0,...,7}
\foreach \y in {0,...,7} 
{

\draw [dashed,fill=gray!20] (\x+0,1+\y)--(\x+1,1+\y)--(\x+1,0+\y)--(\x+0,0+\y)--(\x+0,1+\y);
}

\foreach \x in {10,17}
\foreach \y in {0,...,7} 
{
\draw [dashed,fill=gray!20] (\x+0,1+\y)--(\x+1,1+\y)--(\x+1,0+\y)--(\x+0,0+\y)--(\x+0,1+\y);
}
\foreach \x in {11,...,16}
\foreach \y in {0,7} 
{
\draw [dashed,fill=gray!20] (\x+0,1+\y)--(\x+1,1+\y)--(\x+1,0+\y)--(\x+0,0+\y)--(\x+0,1+\y);
}

\foreach \x in {20,27}
\foreach \y in {0,...,7} 
{
\draw [dashed,fill=gray!20] (\x+0,1+\y)--(\x+1,1+\y)--(\x+1,0+\y)--(\x+0,0+\y)--(\x+0,1+\y);
}
\foreach \x in {21,...,26}
\foreach \y in {0,7} 
{
\draw [dashed,fill=gray!20] (\x+0,1+\y)--(\x+1,1+\y)--(\x+1,0+\y)--(\x+0,0+\y)--(\x+0,1+\y);
}
\foreach \x in {22,...,25}
\foreach \y in {2,5} 
{
\draw [dashed,fill=gray!20] (\x+0,1+\y)--(\x+1,1+\y)--(\x+1,0+\y)--(\x+0,0+\y)--(\x+0,1+\y);
}
\foreach \x in {22,25}
\foreach \y in {3,4} 
{
\draw [dashed,fill=gray!20] (\x+0,1+\y)--(\x+1,1+\y)--(\x+1,0+\y)--(\x+0,0+\y)--(\x+0,1+\y);
}

\foreach \x in {30,...,37}
\foreach \y in {0,1,2,5,6,7} 
{
\draw [dashed,fill=gray!20] (\x+0,1+\y)--(\x+1,1+\y)--(\x+1,0+\y)--(\x+0,0+\y)--(\x+0,1+\y);
}
\foreach \x in {30,31,32,35,36,37}
\foreach \y in {3,4} 
{
\draw [dashed,fill=gray!20] (\x+0,1+\y)--(\x+1,1+\y)--(\x+1,0+\y)--(\x+0,0+\y)--(\x+0,1+\y);
}

\node at (4,-1) {$1$st layer};  \node at (14,-1) {$2$nd layer};  \node at (24,-1) {$3$rd layer};  \node at (34,-1) {$4$th layer};

\end{tikzpicture}
\end{center}
\caption{Layer diagram of $c$, a half functional cube with a dent.}\label{fig_half_c}
\end{figure}

Let $c$ denote the half functional cube with a dent from the top as illustrated in Figure \ref{fig_half_c}, and $C$ denote the half functional cube with a bump from the bottom as illustrated in Figure \ref{fig_half_C}. The dent of $c$ matches the bump of $C$ exactly so that $C$ can be put above $c$ to form a complete $8\times 8\times 8$ functional cube without any gaps or overlaps.  The building blocks $c$ and $C$ will be used to encode the colors of Wang tiles, so they will be referred to as \textit{color} building blocks. In addition to the pair of color building blocks, three more pairs of half functional cubes with either a dent or a bump, $a$ and $A$, $b$ and $B$, and $o$ and $O$, will be used as building blocks in the construction of the set of $6$ polycubes in the next section. Each lowercase letter denotes a half functional cube with a dent, and each uppercase letter denotes a half functional cube with a bump.


\begin{figure}[H]
\begin{center}
\begin{tikzpicture}[scale=0.4]

\foreach \x in {0,10,20,30}
\foreach \y in {0,...,8} 
{ 
\draw [dashed] (\x+0,0+\y)--(\x+8,0+\y);
\draw [dashed] (\x,0)--(\x,8);
\draw [dashed] (\x+8,0)--(\x+8,8);
\draw [dashed] (\x+7,0)--(\x+7,8);
\draw [dashed] (\x+6,0)--(\x+6,8);
\draw [dashed] (\x+5,0)--(\x+5,8);
\draw [dashed] (\x+4,0)--(\x+4,8);
\draw [dashed] (\x+3,0)--(\x+3,8);
\draw [dashed] (\x+2,0)--(\x+2,8);
\draw [dashed] (\x+1,0)--(\x+1,8);
}

\foreach \x in {30,...,37}
\foreach \y in {0,...,7} 
{
\draw [dashed,fill=gray!20] (\x+0,1+\y)--(\x+1,1+\y)--(\x+1,0+\y)--(\x+0,0+\y)--(\x+0,1+\y);
}

\foreach \x in {11,16}
\foreach \y in {1,...,6} 
{
\draw [dashed,fill=gray!20] (\x+0,1+\y)--(\x+1,1+\y)--(\x+1,0+\y)--(\x+0,0+\y)--(\x+0,1+\y);
}
\foreach \x in {12,...,15}
\foreach \y in {1,6} 
{
\draw [dashed,fill=gray!20] (\x+0,1+\y)--(\x+1,1+\y)--(\x+1,0+\y)--(\x+0,0+\y)--(\x+0,1+\y);
}
\foreach \x in {13,14}
\foreach \y in {3,4} 
{
\draw [dashed,fill=gray!20] (\x+0,1+\y)--(\x+1,1+\y)--(\x+1,0+\y)--(\x+0,0+\y)--(\x+0,1+\y);
}

\foreach \x in {23,24}
\foreach \y in {3,4} 
{
\draw [dashed,fill=gray!20] (\x+0,1+\y)--(\x+1,1+\y)--(\x+1,0+\y)--(\x+0,0+\y)--(\x+0,1+\y);
}

\foreach \x in {1,...,6}
\foreach \y in {1,...,6} 
{
\draw [dashed,fill=gray!20] (\x+0,1+\y)--(\x+1,1+\y)--(\x+1,0+\y)--(\x+0,0+\y)--(\x+0,1+\y);
}

\node at (4,-1) {$1$st layer};  \node at (14,-1) {$2$nd layer};  \node at (24,-1) {$3$rd layer};  \node at (34,-1) {$4$th to 7th layers};

\end{tikzpicture}
\end{center}
\caption{Layer diagram of $C$, a half functional cube with a bump.}\label{fig_half_C}
\end{figure}
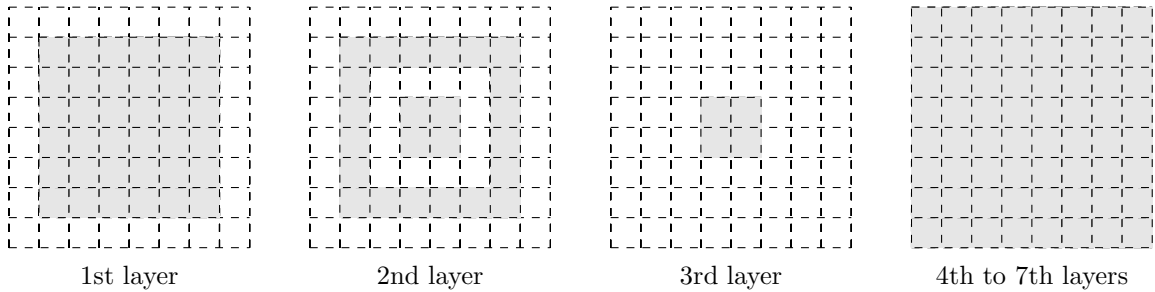


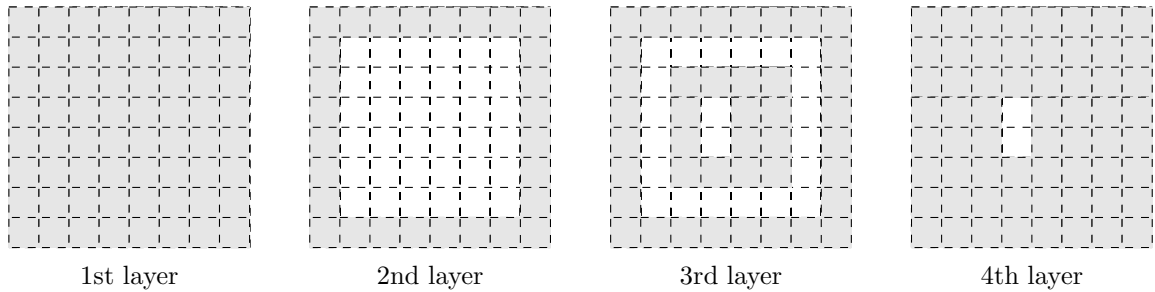
\begin{figure}[H]
\begin{center}
\begin{tikzpicture}[scale=0.4]

\foreach \x in {0,10,20,30}
\foreach \y in {0,...,8} 
{ 
\draw [dashed] (\x+0,0+\y)--(\x+8,0+\y);
\draw [dashed] (\x,0)--(\x,8);
\draw [dashed] (\x+8,0)--(\x+8,8);
\draw [dashed] (\x+7,0)--(\x+7,8);
\draw [dashed] (\x+6,0)--(\x+6,8);
\draw [dashed] (\x+5,0)--(\x+5,8);
\draw [dashed] (\x+4,0)--(\x+4,8);
\draw [dashed] (\x+3,0)--(\x+3,8);
\draw [dashed] (\x+2,0)--(\x+2,8);
\draw [dashed] (\x+1,0)--(\x+1,8);
}

\foreach \x in {0,...,7}
\foreach \y in {0,...,7} 
{

\draw [dashed,fill=gray!20] (\x+0,1+\y)--(\x+1,1+\y)--(\x+1,0+\y)--(\x+0,0+\y)--(\x+0,1+\y);
}

\foreach \x in {10,17}
\foreach \y in {0,...,7} 
{
\draw [dashed,fill=gray!20] (\x+0,1+\y)--(\x+1,1+\y)--(\x+1,0+\y)--(\x+0,0+\y)--(\x+0,1+\y);
}
\foreach \x in {11,...,16}
\foreach \y in {0,7} 
{
\draw [dashed,fill=gray!20] (\x+0,1+\y)--(\x+1,1+\y)--(\x+1,0+\y)--(\x+0,0+\y)--(\x+0,1+\y);
}

\foreach \x in {20,27}
\foreach \y in {0,...,7} 
{
\draw [dashed,fill=gray!20] (\x+0,1+\y)--(\x+1,1+\y)--(\x+1,0+\y)--(\x+0,0+\y)--(\x+0,1+\y);
}
\foreach \x in {21,...,26}
\foreach \y in {0,7} 
{
\draw [dashed,fill=gray!20] (\x+0,1+\y)--(\x+1,1+\y)--(\x+1,0+\y)--(\x+0,0+\y)--(\x+0,1+\y);
}
\foreach \x in {22,...,25}
\foreach \y in {2,5} 
{
\draw [dashed,fill=gray!20] (\x+0,1+\y)--(\x+1,1+\y)--(\x+1,0+\y)--(\x+0,0+\y)--(\x+0,1+\y);
}
\foreach \x in {22,24,25}
\foreach \y in {3,4} 
{
\draw [dashed,fill=gray!20] (\x+0,1+\y)--(\x+1,1+\y)--(\x+1,0+\y)--(\x+0,0+\y)--(\x+0,1+\y);
}

\foreach \x in {30,...,37}
\foreach \y in {0,1,2,5,6,7} 
{
\draw [dashed,fill=gray!20] (\x+0,1+\y)--(\x+1,1+\y)--(\x+1,0+\y)--(\x+0,0+\y)--(\x+0,1+\y);
}
\foreach \x in {30,31,32,34,35,36,37}
\foreach \y in {3,4} 
{
\draw [dashed,fill=gray!20] (\x+0,1+\y)--(\x+1,1+\y)--(\x+1,0+\y)--(\x+0,0+\y)--(\x+0,1+\y);
}

\node at (4,-1) {$1$st layer};  \node at (14,-1) {$2$nd layer};  \node at (24,-1) {$3$rd layer};  \node at (34,-1) {$4$th layer};  

\end{tikzpicture}
\end{center}
\caption{Layer diagram of $a$, a half functional cube with a dent.}\label{fig_half_a}
\end{figure}

Building blocks $a$ and $o$ are illustrated in Figure \ref{fig_half_a} and Figure \ref{fig_half_o}, respectively. Obviously, the building block $o$ can be obtained from $a$ by $90$ degree clockwise rotation about a vertical axis. Let building block $b$ be the polycube obtained from $o$ by performing yet another $90$ degree clockwise rotation about a vertical axis. The building blocks $A$, $B$ and $O$ are half functional cubes with a bump from the top that matches $a$, $b$ and $o$, respectively.


\begin{figure}[H]
\begin{center}
\begin{tikzpicture}[scale=0.4]

\foreach \x in {0,10,20,30}
\foreach \y in {0,...,8} 
{ 
\draw [dashed] (\x+0,0+\y)--(\x+8,0+\y);
\draw [dashed] (\x,0)--(\x,8);
\draw [dashed] (\x+8,0)--(\x+8,8);
\draw [dashed] (\x+7,0)--(\x+7,8);
\draw [dashed] (\x+6,0)--(\x+6,8);
\draw [dashed] (\x+5,0)--(\x+5,8);
\draw [dashed] (\x+4,0)--(\x+4,8);
\draw [dashed] (\x+3,0)--(\x+3,8);
\draw [dashed] (\x+2,0)--(\x+2,8);
\draw [dashed] (\x+1,0)--(\x+1,8);
}

\foreach \x in {0,...,7}
\foreach \y in {0,...,7} 
{

\draw [dashed,fill=gray!20] (\x+0,1+\y)--(\x+1,1+\y)--(\x+1,0+\y)--(\x+0,0+\y)--(\x+0,1+\y);
}

\foreach \x in {10,17}
\foreach \y in {0,...,7} 
{
\draw [dashed,fill=gray!20] (\x+0,1+\y)--(\x+1,1+\y)--(\x+1,0+\y)--(\x+0,0+\y)--(\x+0,1+\y);
}
\foreach \x in {11,...,16}
\foreach \y in {0,7} 
{
\draw [dashed,fill=gray!20] (\x+0,1+\y)--(\x+1,1+\y)--(\x+1,0+\y)--(\x+0,0+\y)--(\x+0,1+\y);
}

\foreach \x in {20,27}
\foreach \y in {0,...,7} 
{
\draw [dashed,fill=gray!20] (\x+0,1+\y)--(\x+1,1+\y)--(\x+1,0+\y)--(\x+0,0+\y)--(\x+0,1+\y);
}
\foreach \x in {21,...,26}
\foreach \y in {0,7} 
{
\draw [dashed,fill=gray!20] (\x+0,1+\y)--(\x+1,1+\y)--(\x+1,0+\y)--(\x+0,0+\y)--(\x+0,1+\y);
}
\foreach \x in {22,...,25}
\foreach \y in {2,3,5} 
{
\draw [dashed,fill=gray!20] (\x+0,1+\y)--(\x+1,1+\y)--(\x+1,0+\y)--(\x+0,0+\y)--(\x+0,1+\y);
}
\foreach \x in {22,25}
\foreach \y in {4} 
{
\draw [dashed,fill=gray!20] (\x+0,1+\y)--(\x+1,1+\y)--(\x+1,0+\y)--(\x+0,0+\y)--(\x+0,1+\y);
}

\foreach \x in {30,...,37}
\foreach \y in {0,1,2,3,5,6,7} 
{
\draw [dashed,fill=gray!20] (\x+0,1+\y)--(\x+1,1+\y)--(\x+1,0+\y)--(\x+0,0+\y)--(\x+0,1+\y);
}
\foreach \x in {30,31,32,35,36,37}
\foreach \y in {4} 
{
\draw [dashed,fill=gray!20] (\x+0,1+\y)--(\x+1,1+\y)--(\x+1,0+\y)--(\x+0,0+\y)--(\x+0,1+\y);
}

\node at (4,-1) {$1$st layer};  \node at (14,-1) {$2$nd layer};  \node at (24,-1) {$3$rd layer};  \node at (34,-1) {$4$th layer};  

\end{tikzpicture}
\end{center}
\caption{Layer diagram of $o$, a half functional cube with a dent.}\label{fig_half_o}
\end{figure}
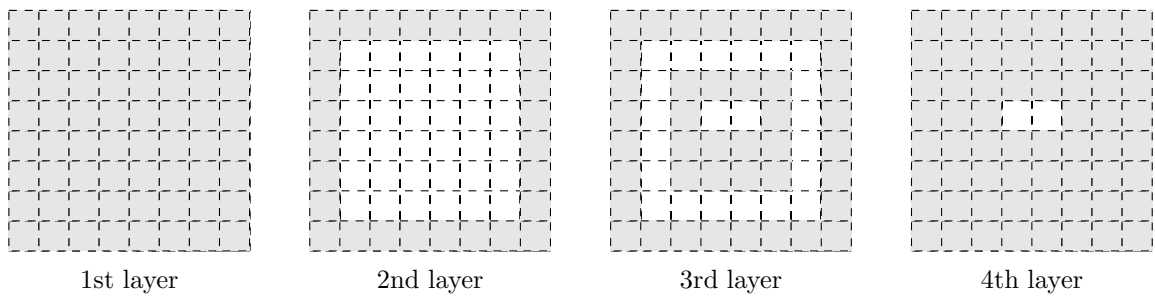

There is a second kind of building block, which is a functional cube with both a dent from the top and a bump of the same shape attached to the bottom. One such building block $\mathbb{M}$ is illustrated in Figure \ref{fig_M}. This type of building block is denoted by uppercase letters in blackboard bold font. There are two more building blocks of this kind, $\mathbb{J}$ and $\mathbb{F}$, which are obtained from $\mathbb{M}$ by a $90$ degree and a $180$ degree clockwise rotation about a vertical axis, respectively.



\begin{figure}[H]
\begin{center}
\begin{tikzpicture}[scale=0.4]

\foreach \x in {0,10,20,30}
\foreach \y in {0,...,8} 
{ 
\draw [dashed] (\x+0,0+\y)--(\x+8,0+\y);
\draw [dashed] (\x,0)--(\x,8);
\draw [dashed] (\x+8,0)--(\x+8,8);
\draw [dashed] (\x+7,0)--(\x+7,8);
\draw [dashed] (\x+6,0)--(\x+6,8);
\draw [dashed] (\x+5,0)--(\x+5,8);
\draw [dashed] (\x+4,0)--(\x+4,8);
\draw [dashed] (\x+3,0)--(\x+3,8);
\draw [dashed] (\x+2,0)--(\x+2,8);
\draw [dashed] (\x+1,0)--(\x+1,8);
}

\foreach \x in {30,...,37}
\foreach \y in {0,...,7} 
{
\draw [dashed,fill=gray!20] (\x+0,1+\y)--(\x+1,1+\y)--(\x+1,0+\y)--(\x+0,0+\y)--(\x+0,1+\y);
}

\foreach \x in {11,16}
\foreach \y in {1,...,6} 
{
\draw [dashed,fill=gray!20] (\x+0,1+\y)--(\x+1,1+\y)--(\x+1,0+\y)--(\x+0,0+\y)--(\x+0,1+\y);
}
\foreach \x in {12,...,15}
\foreach \y in {1,6} 
{
\draw [dashed,fill=gray!20] (\x+0,1+\y)--(\x+1,1+\y)--(\x+1,0+\y)--(\x+0,0+\y)--(\x+0,1+\y);
}
\foreach \x in {13}
\foreach \y in {3} 
{
\draw [dashed,fill=gray!20] (\x+0,1+\y)--(\x+1,1+\y)--(\x+1,0+\y)--(\x+0,0+\y)--(\x+0,1+\y);
}

\foreach \x in {23}
\foreach \y in {3} 
{
\draw [dashed,fill=gray!20] (\x+0,1+\y)--(\x+1,1+\y)--(\x+1,0+\y)--(\x+0,0+\y)--(\x+0,1+\y);
}

\foreach \x in {1,...,6}
\foreach \y in {1,...,6} 
{
\draw [dashed,fill=gray!20] (\x+0,1+\y)--(\x+1,1+\y)--(\x+1,0+\y)--(\x+0,0+\y)--(\x+0,1+\y);
}

\node at (4,-1) {$1$st layer};  \node at (14,-1) {$2$nd layer};  \node at (24,-1) {$3$rd layer};  \node at (34,-1) {$4$th to 8th layers};

\node at (4,-11) {$9$th layer};  \node at (14,-11) {$10$th layer};  \node at (24,-11) {$11$th layer};


\foreach \x in {0,10,20}
\foreach \y in {-10,...,-2} 
{ 
\draw [dashed] (\x+0,0+\y)--(\x+8,0+\y);
\draw [dashed] (\x,-10)--(\x,-2);
\draw [dashed] (\x+8,-10)--(\x+8,-2);
\draw [dashed] (\x+7,-10)--(\x+7,-2);
\draw [dashed] (\x+6,-10)--(\x+6,-2);
\draw [dashed] (\x+5,-10)--(\x+5,-2);
\draw [dashed] (\x+4,-10)--(\x+4,-2);
\draw [dashed] (\x+3,-10)--(\x+3,-2);
\draw [dashed] (\x+2,-10)--(\x+2,-2);
\draw [dashed] (\x+1,-10)--(\x+1,-2);
}

\foreach \x in {0,...,7}
\foreach \y in {-10,-3} 
{
\draw [dashed,fill=gray!20] (\x+0,1+\y)--(\x+1,1+\y)--(\x+1,0+\y)--(\x+0,0+\y)--(\x+0,1+\y);
}

\foreach \x in {0,7}
\foreach \y in {-9,...,-4} 
{
\draw [dashed,fill=gray!20] (\x+0,1+\y)--(\x+1,1+\y)--(\x+1,0+\y)--(\x+0,0+\y)--(\x+0,1+\y);
}

\foreach \x in {10,...,17}
\foreach \y in {-10,-3} 
{
\draw [dashed,fill=gray!20] (\x+0,1+\y)--(\x+1,1+\y)--(\x+1,0+\y)--(\x+0,0+\y)--(\x+0,1+\y);
}

\foreach \x in {10,17}
\foreach \y in {-9,...,-4} 
{
\draw [dashed,fill=gray!20] (\x+0,1+\y)--(\x+1,1+\y)--(\x+1,0+\y)--(\x+0,0+\y)--(\x+0,1+\y);
}

\foreach \x in {12,...,15}
\foreach \y in {-8,-6,-5} 
{
\draw [dashed,fill=gray!20] (\x+0,1+\y)--(\x+1,1+\y)--(\x+1,0+\y)--(\x+0,0+\y)--(\x+0,1+\y);
}

\foreach \x in {12,14,15}
\foreach \y in {-7} 
{
\draw [dashed,fill=gray!20] (\x+0,1+\y)--(\x+1,1+\y)--(\x+1,0+\y)--(\x+0,0+\y)--(\x+0,1+\y);
}

\foreach \x in {20,...,27}
\foreach \y in {-10,-9,-8,-6,-5,-4,-3} 
{
\draw [dashed,fill=gray!20] (\x+0,1+\y)--(\x+1,1+\y)--(\x+1,0+\y)--(\x+0,0+\y)--(\x+0,1+\y);
}

\foreach \x in {20,21,22,24,25,26,27}
\foreach \y in {-7} 
{
\draw [dashed,fill=gray!20] (\x+0,1+\y)--(\x+1,1+\y)--(\x+1,0+\y)--(\x+0,0+\y)--(\x+0,1+\y);
}

\end{tikzpicture}
\end{center}
\caption{Layer diagram of $\mathbb{M}$, a functional cube with both a dent and a bump.}\label{fig_M}
\end{figure}
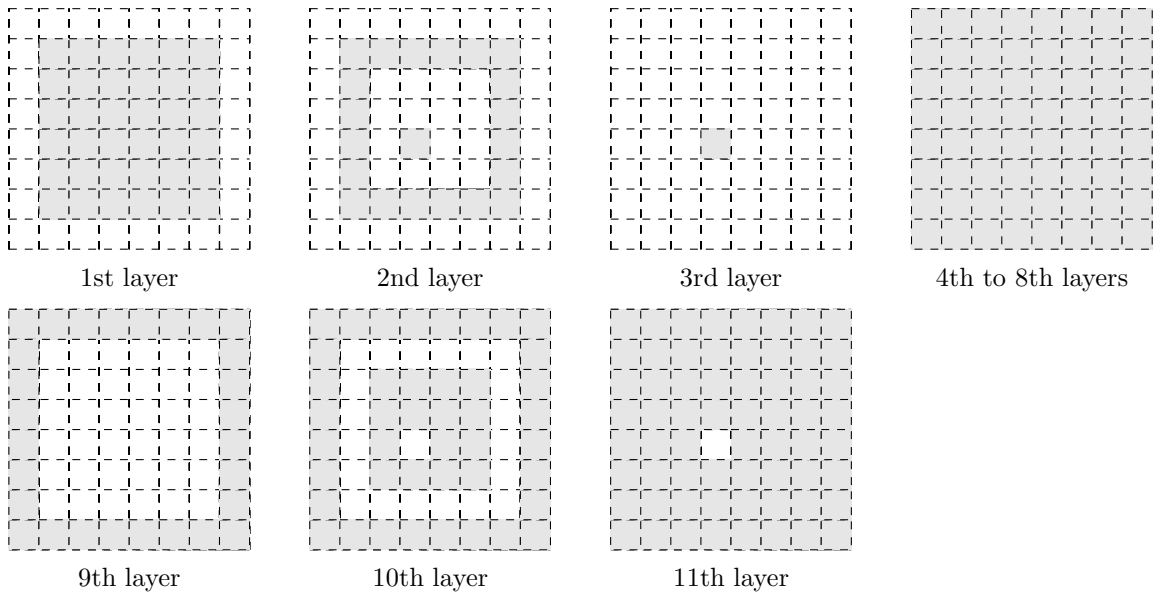

\section{Proof of Theorem \ref{thm_main}}\label{sec_proof}

With the building blocks introduced in the previous section, we are ready to prove our main result.

\begin{proof}[Proof of Theorem \ref{thm_main}] We prove by reduction from Wang's domino problem. Given a set $W$ of Wang tiles, we construct a set $P$ of $6$ polycubes such that there exists a translational tiling of the plane with $W$ if and only if there exists a translational tiling of the $3$-dimensional space with $P$. As mentioned in the first section, we adopt the general framework of Ollinger in his proof of the undecidability of $11$-polyomino tiling problem \cite{o09}. In order to show $6$ polycubes (compared to $11$ in Ollinger's original construction) are sufficient to simulate any set of Wang tiles, we follow more closely to the construction of $8$ polyominoes developed by Yang and Zhang in \cite{yz24}, and introduced novel techniques to further decrease the total number of tiles.

\subsection{The Set of $6$ Polycubes}

We take the set of $3$ Wang tiles illustrated in Figure \ref{fig_wang_set} as an example to describe the construction of the polycubes, and the method can be applied to any set of Wang tiles without any difficulties. To illustrate the construction of the set of $6$ polycubes, we have the following convention in Figure \ref{fig_meat}, Figure \ref{fig_jaw}, Figure \ref{fig_filler} and Figure \ref{fig_link}. These figures are $2$-dimensional top view of the polycubes. Each gray square without a label in the figures represents a normal $8\times 8\times 8$ functional cube. Squares with a label are building blocks that have been introduced in the previous section. So roughly speaking, each polycube in our construction is a thick version of a polyomino (ignoring the dents and bumps).

The first tile is the \textit{meat} (Figure \ref{fig_meat}) which encodes all the Wang tiles in just one polycube. The color building blocks $c$ can be attached to either the upper half or lower half of the main part of the meat polycube, and they are distinguished by adding superscripts or subscripts. A color building block $c^*$ means it is attached to the upper half, and $c_*$ means it is attached to the lower half. See Figure \ref{fig_attach} for a side view of the two ways of the color building blocks $c$ being attached to the main part of a meat polycube. The building blocks $A$, $B$ and $O$ are always attached to the upper half of the main part of the polycube. There is a building block $\mathbb{M}$ at the northwest corner of the meat.


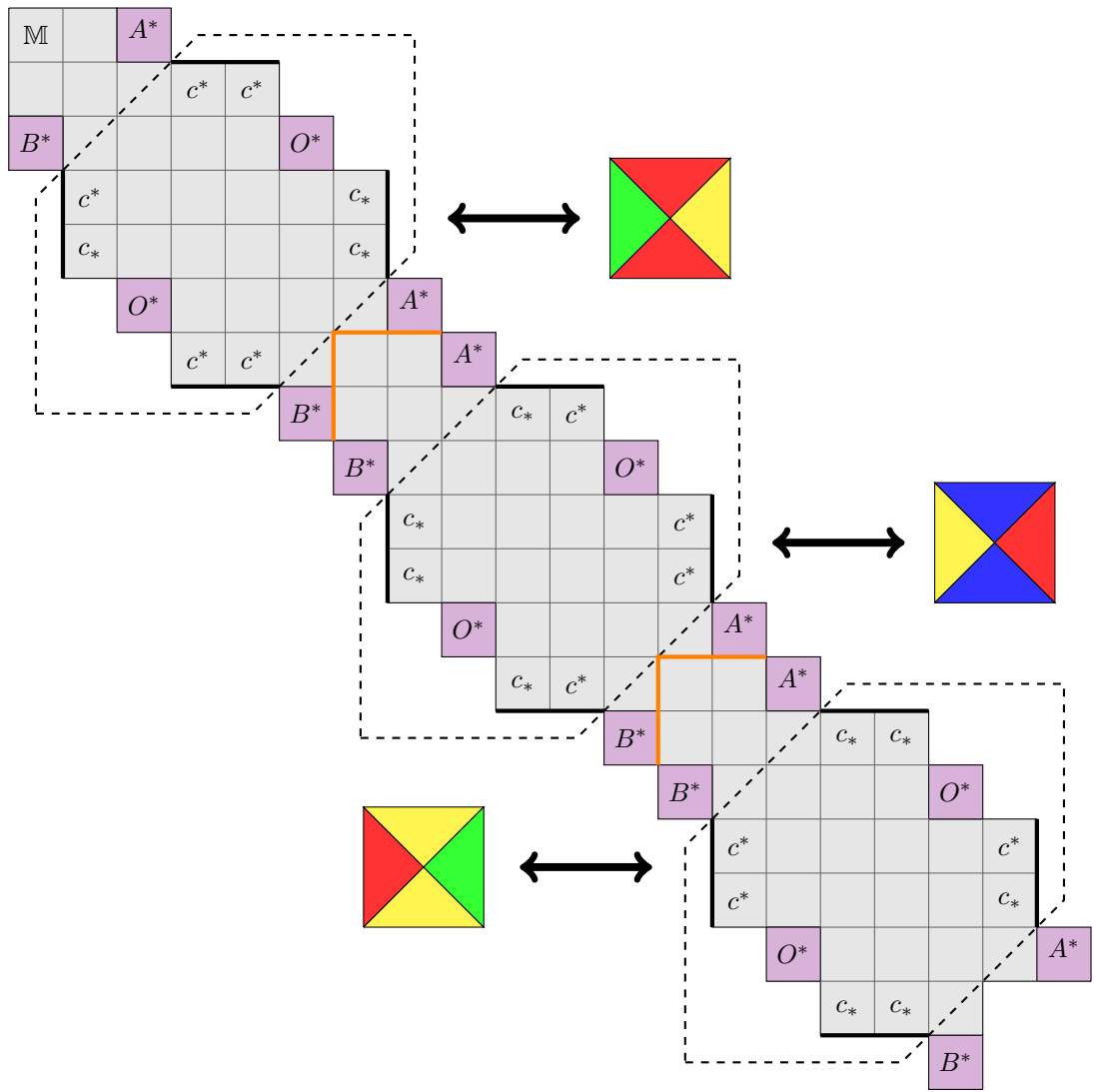
\begin{figure}[H]
\begin{center}
\begin{tikzpicture}[scale=0.08]

\draw [fill=gray!20] (0,63)--(18,63)
--(18,54)
--(45,54)--(45,36)--(63,36)
--(63,18)
--(63,9)--(72,9)
--(72,0)
--(99,0)--(99,-18)--(117,-18)
--(117,-45)
--(126,-45)--(126,-54)
--(153,-54)
--(153,-72)--(171,-72)
--(171,-99)--(162,-99)--(162,-108)
--(135,-108)--(135,-90)--(117,-90)
--(117,-63)--(108,-63)--(108,-54)
--(81,-54)--(81,-36)--(63,-36)
--(63,-9)
--(54,-9)--(54,0)
--(27,0)
--(27,18)--(9,18)
--(9,45)--(0,45)--(0,63);

\node at (4.5, 58.5) {$\mathbb{M}$}; 
\node at (31.5,49.5) {$c^*$}; \node at (40.5,49.5) {$c^*$}; 
\node at (31.5,4.5) {$c^*$}; \node at (40.5,4.5) {$c^*$}; 
\node at (58.5,31.5) {$c_*$}; \node at (58.5,22.5) {$c_*$}; 
\node at (13.5,31.5) {$c^*$}; \node at (13.5,22.5) {$c_*$}; 

\node at (85.5,-4.5) {$c_*$}; \node at (94.5,-4.5) {$c^*$}; 
\node at (85.5,-49.5) {$c_*$}; \node at (94.5,-49.5) {$c^*$}; 
\node at (67.5,-22.5) {$c_*$}; \node at (67.5,-31.5) {$c_*$}; 
\node at (112.5,-22.5) {$c^*$}; \node at (112.5,-31.5) {$c^*$}; 

\node at (139.5,-58.5) {$c_*$}; \node at (148.5,-58.5) {$c_*$}; 
\node at (139.5,-103.5) {$c_*$}; \node at (148.5,-103.5) {$c_*$}; 
\node at (121.5,-76.5) {$c^*$}; \node at (121.5,-85.5) {$c^*$}; 
\node at (166.5,-76.5) {$c^*$}; \node at (166.5,-85.5) {$c_*$}; 


\foreach \x in {0,45,54,99,108,153}
{
\draw [fill=violet!30] (\x+27,54-\x)--(\x+27,63-\x)--(\x+18,63-\x)--(\x+18,54-\x)--(\x+27,54-\x);
\draw [fill=violet!30] (\x+9,36-\x)--(\x+9,45-\x)--(\x+0,45-\x)--(\x+0,36-\x)--(\x+9,36-\x);
\node at (\x+22.5,59.5-\x) {$A^*$};
\node at (\x+4.5,40.5-\x) {$B^*$};
}

\foreach \x in {0,54,108}
{
\draw [fill=violet!30] (\x+54,36-\x)--(\x+54,45-\x)--(\x+45,45-\x)--(\x+45,36-\x)--(\x+54,36-\x);
\node at (\x+49.5,40.5-\x) {$O^*$};

\draw [fill=violet!30] (\x+27,9-\x)--(\x+27,18-\x)--(\x+18,18-\x)--(\x+18,9-\x)--(\x+27,9-\x);
\node at (\x+22.5,13.5-\x) {$O^*$};
}

\draw [color=black!60] (0,54)--(18,54)--(18,18);
\draw [color=black!60] (9,63)--(9,45)--(45,45);
\draw [color=black!60] (9,36)--(45,36)--(45,0);
\draw [color=black!60] (9,27)--(63,27);
\draw [color=black!60] (27,54)--(27,18)--(63,18);
\draw [color=black!60] (36,54)--(36,0);
\draw [color=black!60] (27,9)--(63,9)--(63,-9)--(99,-9);
\draw [color=black!60] (54,36)--(54,0)--(72,0)--(72,-36);
\draw [color=black!60] (81,0)--(81,-36)--(117,-36);
\draw [color=black!60] (63,-18)--(99,-18)--(99,-54);
\draw [color=black!60] (63,-27)--(117,-27);
\draw [color=black!60] (90,0)--(90,-54);
\draw [color=black!60] (81,-45)--(117,-45)--(117,-63)--(153,-63);
\draw [color=black!60] (108,-18)--(108,-54)--(126,-54)--(126,-90);
\draw [color=black!60] (117,-72)--(153,-72)--(153,-108);
\draw [color=black!60] (117,-81)--(171,-81);
\draw [color=black!60] (135,-54)--(135,-90)--(171,-90);
\draw [color=black!60] (144,-54)--(144,-108);
\draw [color=black!60] (135,-99)--(162,-99)--(162,-72);

\draw [dashed, thick] (4.5,-4.5)--(40.5,-4.5)--(67.5,22.5)--(67.5,58.5)--(31.5,58.5)--(4.5,31.5)--(4.5,-4.5);

\foreach \x in {100}
\foreach \y in {18}
{
\draw [fill=green!80] (\x+0,0+\y)--(\x+10,10+\y)--(\x+0,20+\y)--(\x+0,0+\y);
\draw [fill=red!80] (\x+0,0+\y)--(\x+20,0+\y)--(\x+0,20+\y)--(\x+20,20+\y)--(\x+0,0+\y);
\draw [fill=yellow!80] (\x+10,10+\y)--(\x+20,20+\y)--(\x+20,0+\y)--(\x+10,10+\y);
}
\draw  [<->, line width=3,  black] (73,28)--(95,28);

\draw [dashed, thick] (54+4.5,-4.5-54)--(54+40.5,-4.5-54)--(54+67.5,22.5-54)--(54+67.5,58.5-54)--(54+31.5,58.5-54)--(54+4.5,31.5-54)--(54+4.5,-4.5-54);

\foreach \x in {154}
\foreach \y in {-36}
{
\draw [fill=yellow!80] (\x+0,0+\y)--(\x+10,10+\y)--(\x+0,20+\y)--(\x+0,0+\y);
\draw [fill=blue!80] (\x+0,0+\y)--(\x+20,0+\y)--(\x+0,20+\y)--(\x+20,20+\y)--(\x+0,0+\y);
\draw [fill=red!80] (\x+10,10+\y)--(\x+20,20+\y)--(\x+20,0+\y)--(\x+10,10+\y);
}

\draw  [<->, line width=3,  black] (73+54,-26)--(95+54,-26);

\draw [dashed, thick] (108+4.5,-4.5-108)--(108+40.5,-4.5-108)--(108+67.5,22.5-108)--(108+67.5,58.5-108)--(108+31.5,58.5-108)--(108+4.5,31.5-108)--(108+4.5,-4.5-108);
\draw  [<->, line width=3,  black] (73+12,-80)--(95+12,-80);

\foreach \x in {59}
\foreach \y in {-90}
{
\draw [fill=red!80] (\x+0,0+\y)--(\x+10,10+\y)--(\x+0,20+\y)--(\x+0,0+\y);
\draw [fill=yellow!80] (\x+0,0+\y)--(\x+20,0+\y)--(\x+0,20+\y)--(\x+20,20+\y)--(\x+0,0+\y);
\draw [fill=green!80] (\x+10,10+\y)--(\x+20,20+\y)--(\x+20,0+\y)--(\x+10,10+\y);
}

\foreach \x in {0}
\foreach \y in {0}
{
\draw [ultra thick] (\x+27,\y)--(\x+45,\y); \draw [ultra thick] (\x+27,54+\y)--(\x+45,54+\y);
\draw [ultra thick] (\x+9,18+\y)--(\x+9,36+\y); \draw [ultra thick] (\x+63,18+\y)--(\x+63,36+\y); 

\draw [color=orange, ultra thick] (\x+54,-9+\y)--(\x+54,9+\y)--(\x+72,9+\y);  
}
\foreach \x in {54}
\foreach \y in {-54}
{
\draw [ultra thick] (\x+27,\y)--(\x+45,\y); \draw [ultra thick] (\x+27,54+\y)--(\x+45,54+\y);
\draw [ultra thick] (\x+9,18+\y)--(\x+9,36+\y); \draw [ultra thick] (\x+63,18+\y)--(\x+63,36+\y); 
\draw [color=orange, ultra thick] (\x+54,-9+\y)--(\x+54,9+\y)--(\x+72,9+\y);  
}\foreach \x in {108}
\foreach \y in {-108}
{
\draw [ultra thick] (\x+27,\y)--(\x+45,\y); \draw [ultra thick] (\x+27,54+\y)--(\x+45,54+\y);
\draw [ultra thick] (\x+9,18+\y)--(\x+9,36+\y); \draw [ultra thick] (\x+63,18+\y)--(\x+63,36+\y); 
}

\end{tikzpicture}
\end{center}
\caption{The meat.}\label{fig_meat}
\end{figure}

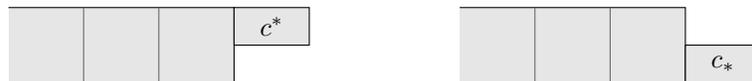
\begin{figure}[H]
\begin{center}
\begin{tikzpicture}

\draw [fill=gray!20] (0,0)--(3,0)--(3,1)--(0,1); \draw [fill=gray!20] (3,0.5)--(4,0.5)--(4,1)--(3,1)--(3,0.5);
\draw [fill=gray!20] (6+0,0)--(6+3,0)--(6+3,1)--(6+0,1); \draw [fill=gray!20] (6+3,0.5)--(6+4,0.5)--(6+4,0)--(6+3,0)--(6+3,0.5);

\foreach \x in {1,2,7,8}  
{ 
\draw [ color=black!60] (\x,0)--(\x,1); 
}

\node at (3.5,0.75) {$c^*$}; \node at (9.5,0.25) {$c_*$};

\end{tikzpicture}
\end{center}
\caption{Building blocks attached to the main part of a polycube (side view).}\label{fig_attach}
\end{figure}


The three Wang tiles of Figure \ref{fig_wang_set} are simulated in parts of the meat polycube enclosed by dashed lines in Figure \ref{fig_meat}. These parts will be referred to \textit{simulated Wang tiles}. The positions (upper half or lower half) that the building blocks $c$ are attached to the main part of the meat form a binary encoding system for the colors of the Wang tiles. In this example, $c^*c^*$, $c^*c_*$, $c_*c^*$ and $c_*c_*$ encode the colors red, green, blue and yellow, respectively. Given an arbitrary set of Wang tiles, the size of a simulated Wang tile in the meat polycube can increase if we need more binary bits to encode more different colors. In other words, the length of the thick black lines shown in Figure \ref{fig_meat} can be elongated to include more building blocks $c$. Note also that the thick orange lines in Figure \ref{fig_meat} divide the meat into $3$ \textit{segments}, which are almost identical except for the building block $\mathbb{M}$ and the positions of the color building blocks. In general, to simulate a set of $k$ Wang tiles, we construct a meat polycube with $k$ segments.

The top view of the second polycube, the \textit{jaw}, is illustrated at the top of Figure \ref{fig_jaw}. The jaw has two concave \textit{mouths}, one at the northwest corner and the other southeast. There are several building blocks $a$, $b$ and $o$ being attached inside the mouths of the jaw. Note also that there is a building block $o$ attached to the northeast corner and southwest corner of the jaw, respectively. All the building blocks $a$, $b$ and $o$ are attached to the lower half of the main part of the jaw. The concave shape of the mouths matches the shape of the meat so that a meat polycube can be put partially inside a mouth of the jaw polycube. But the mouth can grip at most all but one simulated a Wang tile from one side of the meat, always leaving some parts of the meat outside the mouth. In general, for a set of $k$ Wang tiles, we construct a jaw polycube that can grip at most $k-1$ segments of the meat by one of its mouths. There is a building block $\mathbb{J}$ near the center of the jaw.


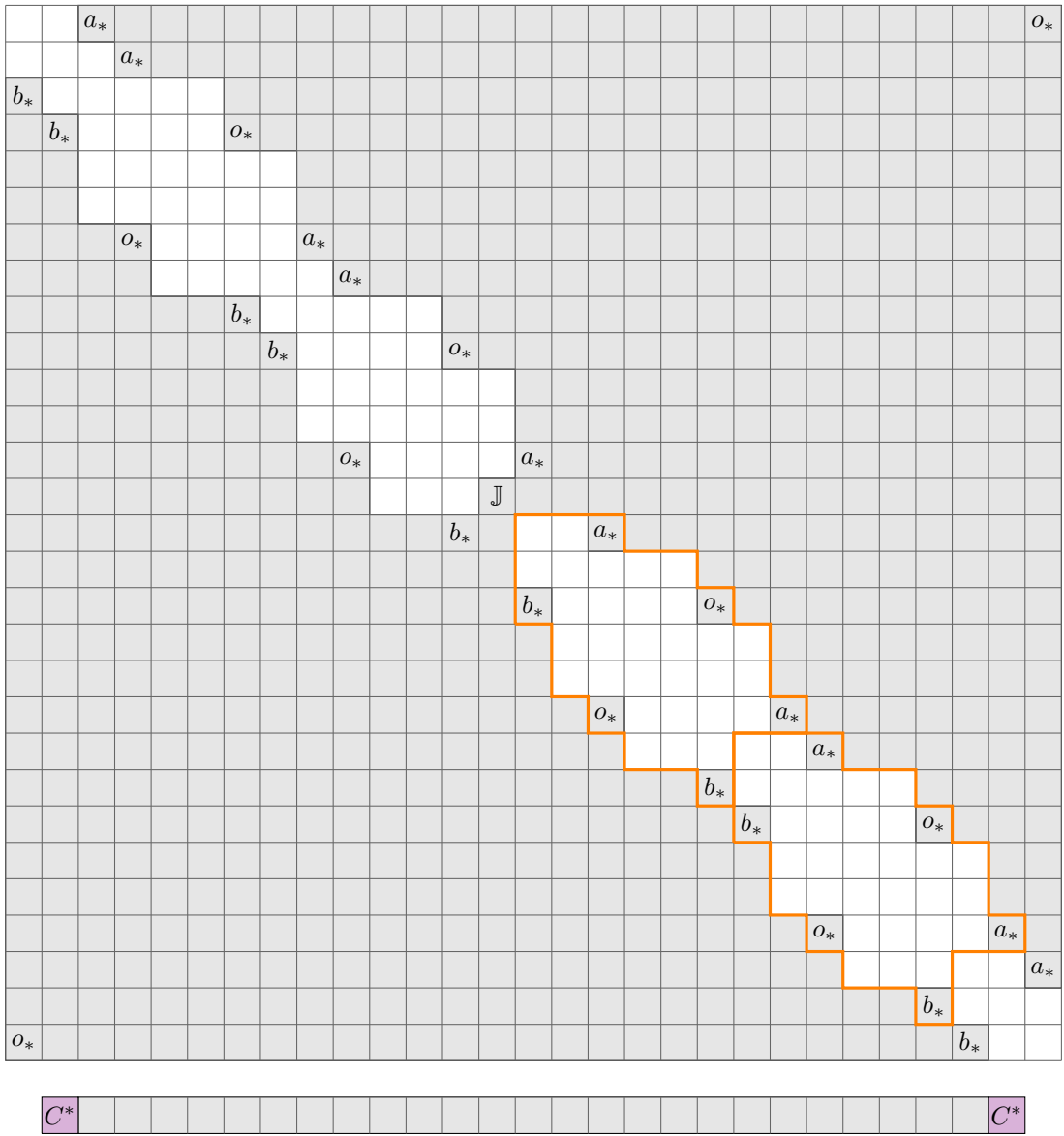
\begin{figure}[H]
\begin{center}
\begin{tikzpicture}[scale=0.5]

\draw [fill=gray!20] (0,0)--(27,0)
--(27,1)--(26,1)--(26,2)--(23,2)--(23,4)--(21,4)--(21,7)--(20,7)--(20,8)--(17,8)--(17,10)--(15,10)--(15,13)--(14,13)--(14,15)--(16,15)--(16,14)--(19,14)--(19,12)--(21,12)--(21,9)--(22,9)--(22,8)--(25,8)--(25,6)--(27,6)--(27,3)--(28,3)--(28,2)
--(29,2)
--(29,29)
--(2,29)
--(2,28)--(3,28)--(3,27)--(6,27)--(6,25)--(8,25)--(8,22)--(9,22)--(9,21)--(12,21)--(12,19)--(14,19)--(14,16)--(13,16)--(13,15)--(10,15)--(10,17)--(8,17)--(8,20)--(7,20)--(7,21)--(4,21)--(4,23)--(2,23)--(2,26)--(1,26)--(1,27)
--(0,27)--(0,0);

\foreach \x in {0,...,29}  
{ 
\draw [ color=black!60] (\x,0)--(\x,29);
\draw [ color=black!60] (0,\x)--(29,\x);
}

\foreach \x in {0,1,6,7,12,14,19,20,25,26}
{
\node at (\x+2.5,28.5-\x) {$a_*$};
\node at (\x+0.5,26.5-\x) {$b_*$};
}

\foreach \x in {0,6,13,19}
{
\node at (\x+6.5,25.5-\x) {$o_*$};
\node at (\x+3.5,22.5-\x) {$o_*$};
}
\node at (0.5,0.5) {$o_*$}; \node at (28.5,28.5) {$o_*$};

\node at (13.5,15.5) {$\mathbb{J}$};


\draw [fill=gray!20] (2,-2)--(27,-2)--(27,-1)--(2,-1)--(2,-2);
\draw [fill=violet!30] (27,-2)--(28,-2)--(28,-1)--(27,-1)--(27,-2);

\draw [fill=violet!30] (1,-2)--(2,-2)--(2,-1)--(1,-1)--(1,-2);
\node at (1.5,-1.5) {$C^*$};
\node at (27.5,-1.5) {$C^*$};

\foreach \x in{3,...,26}
{
\draw [color=black!60] (\x,-2)--(\x,-1);
}

\draw [orange, very thick] (14,12)--(14,15)--(17,15)--(17,14)--(19,14)--(19,13)--(20,13)--(20,12)--(21,12)--(21,10)--(22,10)--(22,9)--(20,9)--(20,7)--(19,7)--(19,8)--(17,8)--(17,9)--(16,9)--(16,10)--(15,10)--(15,12)--(14,12);

\foreach \x in{6}
\foreach \y in{-6}
{
\draw [orange, very thick](\x+14,12+\y)--(\x+14,15+\y)--(\x+17,15+\y)--(\x+17,14+\y)--(\x+19,14+\y)--(\x+19,13+\y)--(\x+20,13+\y)--(\x+20,12+\y)--(\x+21,12+\y)--(\x+21,10+\y)--(\x+22,10+\y)--(\x+22,9+\y)--(\x+20,9+\y)--(\x+20,7+\y)--(\x+19,7+\y)--(\x+19,8+\y)--(\x+17,8+\y)--(\x+17,9+\y)--(\x+16,9+\y)--(\x+16,10+\y)--(\x+15,10+\y)--(\x+15,12+\y)--(\x+14,12+\y);
}

\end{tikzpicture}
\end{center}
\caption{The jaw (with a link at the bottom for comparison).}\label{fig_jaw}
\end{figure}

The spaces unoccupied by the meat polycubes inside the mouths of the jaw can be filled by the \textit{filler} polycube as illustrated in Figure \ref{fig_filler}. It is almost the the same as a segment of the meat except that the color building blocks are replaced by normal $8\times 8\times 8$ polycubes and there is a building block $\mathbb{F}$ at the northwest corner. The building blocks $A$, $B$ and $O$ are attached to the upper half of the main part of the filler, just like those of the meat. The orange lines in Figure \ref{fig_jaw} show how the fillers or segments of a meat can fit inside the mouths of a jaw.

Note that there are building blocks $\mathbb{M}$, $\mathbb{J}$ and $\mathbb{F}$ in the meat, jaw and filler polycubes, respectively. These building blocks ensure another polycube of the same type must be placed directly under or above a meat, a jaw and a filler in order to tile the entire $3$-dimensional space.


\begin{figure}[H]
\begin{center}
\begin{tikzpicture}[scale=0.1]

\draw [fill=gray!20] (0,63)--(18,63)
--(18,54)
--(45,54)--(45,36)--(63,36)
--(63,18)
--(63,9)--(54,9)--(54,0)
--(45,0)
--(27,0)
--(27,18)--(9,18)
--(9,36)
--(9,45)--(0,45)--(0,63);

\draw [color=black!60] (0,54)--(18,54)--(18,18);
\draw [color=black!60] (9,63)--(9,45)--(45,45);
\draw [color=black!60] (9,36)--(45,36)--(45,0);
\draw [color=black!60] (9,27)--(63,27);
\draw [color=black!60] (27,54)--(27,18)--(63,18);
\draw [color=black!60] (36,54)--(36,0);
\draw [color=black!60] (27,9)--(54,9)--(54,36);
\draw [color=black!60] (9,18)--(9,36);
\draw [color=black!60] (63,18)--(63,36);
\draw [color=black!60] (27,0)--(45,0);
\draw [color=black!60] (27,54)--(45,54);

\foreach \x in {0,45}
{
\draw [fill=violet!30] (\x+27,54-\x)--(\x+27,63-\x)--(\x+18,63-\x)--(\x+18,54-\x)--(\x+27,54-\x);
\draw [fill=violet!30] (\x+9,36-\x)--(\x+9,45-\x)--(\x+0,45-\x)--(\x+0,36-\x)--(\x+9,36-\x);
\node at (\x+22.5,59.5-\x) {$A^*$};
\node at (\x+4.5,40.5-\x) {$B^*$};
}

\foreach \x in {0}
{
\draw [fill=violet!30] (\x+54,36-\x)--(\x+54,45-\x)--(\x+45,45-\x)--(\x+45,36-\x)--(\x+54,36-\x);
\node at (\x+49.5,40.5-\x) {$O^*$};

\draw [fill=violet!30] (\x+27,9-\x)--(\x+27,18-\x)--(\x+18,18-\x)--(\x+18,9-\x)--(\x+27,9-\x);
\node at (\x+22.5,13.5-\x) {$O^*$};
}

\node at (4.5,58.5) {$\mathbb{F}$};

\end{tikzpicture}
\end{center}
\caption{A filler.}\label{fig_filler}
\end{figure}
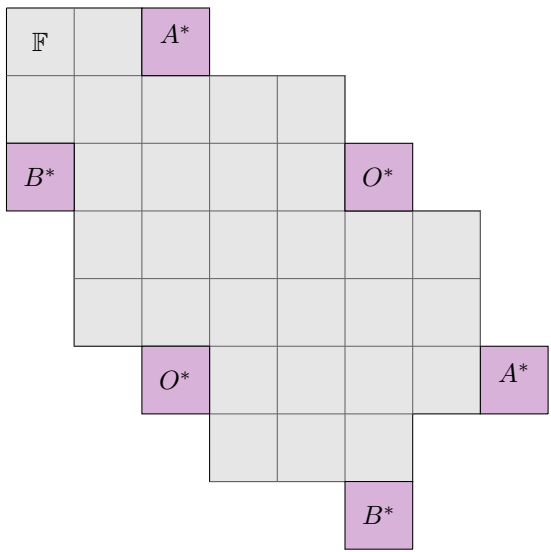

The fourth polycube, the \textit{tooth} (see the right of Figure \ref{fig_link}), is just a building block $C$, which is a half functional cube with a bump. The fifth and sixth polycubes are the \textit{links} which are identical except for the orientation. The east-west oriented link is illustrated on the left of Figure \ref{fig_link}. The south-north oriented link is obtained from the east-west link by a $90$ degree rotation about a vertical axis. There are two building blocks $C$ attached to the upper half of the two ends of the link. As a result, a link can only connect two building blocks $c$ (of two meat polycubes) at the same altitude. And the length of the link must be set properly in order to connect two building blocks. See the bottom of Figure \ref{fig_jaw} for a comparison of the lengths of the link and the jaw. In general, the link is shorter than the jaw by two normal $8\times 8\times 8$ functional cubes.

To summarize, we have constructed a set of $6$ polycubes: a meat, a jaw, a filler, a tooth and two links. As we have mentioned in the previous paragraphs, the sizes of the meat, the jaw, the filler and the links depend on the number of tiles and colors of the given set of Wang tiles.


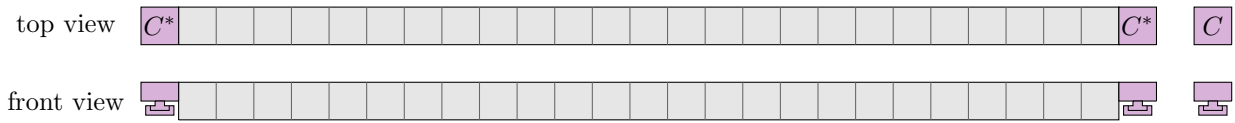
\begin{figure}[H]
\begin{center}
\begin{tikzpicture}[scale=0.5]

\draw [fill=gray!20] (1,0)--(26,0)--(26,1)--(1,1)--(1,0);

\draw [fill=gray!20] (1,-2)--(26,-2)--(26,-1)--(1,-1)--(1,-2);
\draw [fill=violet!30] (27,-1.5)--(26,-1.5)--(26,-1)--(27,-1)--(27,-1.5);

\draw [fill=violet!30] (29,-1.5)--(28,-1.5)--(28,-1)--(29,-1)--(29,-1.5);

\draw [fill=violet!30] (1,-1.5)--(0,-1.5)--(0,-1)--(1,-1)--(1,-1.5);
\node at (-2,-1.5) {front view}; \node at (-2, 0.5) {top view};

\foreach \x in {0,26,28}
{
\draw [fill=violet!30] (\x+0.375,-1.5)--(\x+0.375,-1.625)--(\x+0.125,-1.625)--(\x+0.125,-1.875)--(\x+0.875,-1.875)--(\x+0.875,-1.625)--(\x+0.625,-1.625)--(\x+0.625,-1.5)--(\x+0.375,-1.5); 
\draw (\x+0.25,-1.625)--(\x+0.25,-1.75)--(\x+0.75,-1.75)--(\x+0.75,-1.625);
}

\foreach \x in{2,...,25}
{
\draw [color=black!60] (\x,0)--(\x,1);
\draw [color=black!60] (\x,-2)--(\x,-1);
}

\draw [fill=violet!30] (1,0)--(0,0)--(0,1)--(1,1)--(1,0);\draw [fill=violet!30] (27,0)--(26,0)--(26,1)--(27,1)--(27,0);
\node at (0.5,0.5) {$C^*$};
\node at (26.5,0.5) {$C^*$};


\draw [fill=violet!30] (29,0)--(28,0)--(28,1)--(29,1)--(29,0);
\node at (28.5,0.5) {$C$};

\end{tikzpicture}
\end{center}
\caption{A east-west link and a tooth.}\label{fig_link}
\end{figure}

\subsection{Tiling the Space}

To complete the proof of Theorem \ref{thm_main}, we will show that the only way to tile the $3$-dimensional space with the set of $6$ polycubes we just constructed in the previous subsection is to simulate a tiling of the plane with the corresponding set of Wang tiles.

\begin{itemize}
    \item First of all, we claim that the meat polycubes must be used in any tiling of the entire space. If the tooth or the link is used, then the meat must be used as the only building blocks that can match the building blocks $C$ in the tooth or link are the building blocks $c$ of the meat. If the fillers are used, then in order to match the building blocks $A$, $B$ or $O$ in the filler, the jaws must be used too. Using only fillers and jaws, by the same arguments we will see soon in the next paragraph, we must form a pattern of tiling almost the same as that illustrated in Figure \ref{fig_pattern} except that each orange meat polycube is replaced by $3$ green fillers. Then there is nothing to fill the gaps outside the jaws as the links are longer than the gaps and cannot be squashed into the gaps. So the meat polycubes must be used.


\begin{figure}[H]
\begin{center}
\begin{tikzpicture}[scale=0.1]

\foreach \x in {0,...,3}
\foreach \y in {0,...,3} 
{ 
\draw [ fill=gray!20] (31*\x+0,0+31*\y)--(31*\x+0,29+31*\y)--(31*\x+29,29+31*\y)--(31*\x+29,0+31*\y)--(31*\x+0,0+31*\y);
}

\foreach \x in {-1,...,3}
\foreach \y in {-1,...,3} 
{ 
\draw [fill=orange] (31*\x+28,29+31*\y)--(31*\x+28,28+31*\y)--(31*\x+29,28+31*\y)--(31*\x+29,27+31*\y)--(31*\x+31,27+31*\y)--(31*\x+31,29+31*\y)--(31*\x+33,29+31*\y)--(31*\x+33,31+31*\y)--(31*\x+32,31+31*\y)--(31*\x+32,32+31*\y)--(31*\x+31,32+31*\y)--(31*\x+31,33+31*\y)--(31*\x+29,33+31*\y)--(31*\x+29,31+31*\y)--(31*\x+27,31+31*\y)--(31*\x+27,29+31*\y)--(31*\x+28,29+31*\y);

\node at (31*\x+27.5,30.5+31*\y) {$\cdot$}; \node at (31*\x+27.5,29.5+31*\y) {$\cdot$}; \node at (31*\x+32.5,30.5+31*\y) {$\cdot$}; \node at (31*\x+32.5,29.5+31*\y) {$\cdot$};
\node at (31*\x+29.5,27.5+31*\y) {$\cdot$}; \node at (31*\x+30.5,27.5+31*\y) {$\cdot$}; \node at (31*\x+29.5,32.5+31*\y) {$\cdot$}; \node at (31*\x+30.5,32.5+31*\y) {$\cdot$};
}

\foreach \x in {0,31,62,93}
\foreach \y in {0,1,31,32,62,63} 
{ 
\draw (\x+2,29+\y)--(\x+27,29+\y)--(\x+27,30+\y)--(\x+2,30+\y)--(\x+2,29+\y);
}

\foreach \y in {0,31,62,93}
\foreach \x in {0,1,31,32,62,63} 
{ 
\draw (\x+29,2+\y)--(\x+29,27+\y)--(\x+30,27+\y)--(\x+30,2+\y)--(\x+29,2+\y);
}


\foreach \x in {0}
\foreach \y in {0} 
{ 
\draw [fill=orange] (\x+59,60+\y)--(\x+59,59+\y)--(\x+60,59+\y)--(\x+60,58+\y)--(\x+62,58+\y)--(\x+62,57+\y)--(\x+63,57+\y)--(\x+63,56+\y)--(\x+64,56+\y)--(\x+64,54+\y)--(\x+65,54+\y)--(\x+65,53+\y)--(\x+66,53+\y)--(\x+66,52+\y)--(\x+68,52+\y)--(\x+68,51+\y)--(\x+69,51+\y)--(\x+69,53+\y)--(\x+71,53+\y)--(\x+71,54+\y)--(\x+70,54+\y)--(\x+70,56+\y)--(\x+69,56+\y)--(\x+69,57+\y)--(\x+68,57+\y)--(\x+68,58+\y)--(\x+66,58+\y)--(\x+66,59+\y)--(\x+65,59+\y)--(\x+65,60+\y)--(\x+64,60+\y)--(\x+64,62+\y)--(\x+63,62+\y)--(\x+63,63+\y)
--(\x+62,63+\y)--(\x+62,64+\y)--(\x+60,64+\y)--(\x+60,65+\y)--(\x+59,65+\y)--(\x+59,66+\y)--(\x+58,66+\y)--(\x+58,68+\y)--(\x+57,68+\y)--(\x+57,69+\y)--(\x+56,69+\y)--(\x+56,70+\y)--(\x+54,70+\y)--(\x+54,71+\y)--(\x+53,71+\y)--(\x+51,71+\y)--(\x+51,68+\y)--(\x+52,68+\y)--(\x+52,66+\y)--(\x+53,66+\y)--(\x+53,65+\y)--(\x+54,65+\y)--(\x+54,64+\y)--(\x+56,64+\y)--(\x+56,63+\y)--(\x+57,63+\y)--(\x+57,62+\y)--(\x+58,62+\y)--(\x+58,60+\y)--(\x+59,60+\y);
\draw [color=black!50] (\x+57,63+\y)--(\x+57,65+\y)--(\x+59,65+\y);
\draw [color=black!50] (\x+63,57+\y)--(\x+63,59+\y)--(\x+65,59+\y);
\node at (60.5,58.5) {$\cdot$}; \node at (61.5,58.5) {$\cdot$}; \node at (64.5,55.5) {$\cdot$}; \node at (64.5,54.5) {$\cdot$};
\node at (66.5,52.5) {$\cdot$}; \node at (67.5,52.5) {$\cdot$}; \node at (69.5,55.5) {$\cdot$}; \node at (69.5,54.5) {$\cdot$};
\node at (66.5,57.5) {$\cdot$}; \node at (67.5,57.5) {$\cdot$}; \node at (63.5,61.5) {$\cdot$}; \node at (63.5,60.5) {$\cdot$};
\node at (60.5,63.5) {$\cdot$}; \node at (61.5,63.5) {$\cdot$}; \node at (57.5,67.5) {$\cdot$}; \node at (57.5,66.5) {$\cdot$};
\node at (54.5,69.5) {$\cdot$}; \node at (55.5,69.5) {$\cdot$}; \node at (52.5,67.5) {$\cdot$}; \node at (52.5,66.5) {$\cdot$}; 
\node at (54.5,64.5) {$\cdot$}; \node at (55.5,64.5) {$\cdot$}; \node at (58.5,61.5) {$\cdot$}; \node at (58.5,60.5) {$\cdot$};
}

\foreach \x in {37}
\foreach \y in {-37} 
{ 
\draw [fill=orange] (\x+59,60+\y)--(\x+59,59+\y)--(\x+60,59+\y)--(\x+60,58+\y)--(\x+62,58+\y)--(\x+62,57+\y)--(\x+63,57+\y)--(\x+63,56+\y)--(\x+64,56+\y)--(\x+64,54+\y)--(\x+65,54+\y)--(\x+65,53+\y)--(\x+66,53+\y)--(\x+66,52+\y)--(\x+68,52+\y)--(\x+68,51+\y)--(\x+69,51+\y)--(\x+69,53+\y)--(\x+71,53+\y)--(\x+71,54+\y)--(\x+70,54+\y)--(\x+70,56+\y)--(\x+69,56+\y)--(\x+69,57+\y)--(\x+68,57+\y)--(\x+68,58+\y)--(\x+66,58+\y)--(\x+66,59+\y)--(\x+65,59+\y)--(\x+65,60+\y)--(\x+64,60+\y)--(\x+64,62+\y)--(\x+63,62+\y)--(\x+63,63+\y)
--(\x+62,63+\y)--(\x+62,64+\y)--(\x+60,64+\y)--(\x+60,65+\y)--(\x+59,65+\y)--(\x+59,66+\y)--(\x+58,66+\y)--(\x+58,68+\y)--(\x+57,68+\y)--(\x+57,69+\y)--(\x+56,69+\y)--(\x+56,70+\y)--(\x+54,70+\y)--(\x+54,71+\y)--(\x+53,71+\y)--(\x+51,71+\y)--(\x+51,68+\y)--(\x+52,68+\y)--(\x+52,66+\y)--(\x+53,66+\y)--(\x+53,65+\y)--(\x+54,65+\y)--(\x+54,64+\y)--(\x+56,64+\y)--(\x+56,63+\y)--(\x+57,63+\y)--(\x+57,62+\y)--(\x+58,62+\y)--(\x+58,60+\y)--(\x+59,60+\y);
\draw [color=black!50] (\x+57,63+\y)--(\x+57,65+\y)--(\x+59,65+\y);
\draw [color=black!50] (\x+63,57+\y)--(\x+63,59+\y)--(\x+65,59+\y);
\node at (\x+60.5,58.5+\y) {$\cdot$}; \node at (\x+61.5,58.5+\y) {$\cdot$}; \node at (\x+64.5,55.5+\y) {$\cdot$}; \node at (\x+64.5,54.5+\y) {$\cdot$};
\node at (\x+66.5,52.5+\y) {$\cdot$}; \node at (\x+67.5,52.5+\y) {$\cdot$}; \node at (\x+69.5,55.5+\y) {$\cdot$}; \node at (\x+69.5,54.5+\y) {$\cdot$};
\node at (\x+66.5,57.5+\y) {$\cdot$}; \node at (\x+67.5,57.5+\y) {$\cdot$}; \node at (\x+63.5,61.5+\y) {$\cdot$}; \node at (\x+63.5,60.5+\y) {$\cdot$};
\node at (\x+60.5,63.5+\y) {$\cdot$}; \node at (\x+61.5,63.5+\y) {$\cdot$}; \node at (\x+57.5,67.5+\y) {$\cdot$}; \node at (\x+57.5,66.5+\y) {$\cdot$};
\node at (\x+54.5,69.5+\y) {$\cdot$}; \node at (\x+55.5,69.5+\y) {$\cdot$}; \node at (\x+52.5,67.5+\y) {$\cdot$}; \node at (\x+52.5,66.5+\y) {$\cdot$}; 
\node at (\x+54.5,64.5+\y) {$\cdot$}; \node at (\x+55.5,64.5+\y) {$\cdot$}; \node at (\x+58.5,61.5+\y) {$\cdot$}; \node at (\x+58.5,60.5+\y) {$\cdot$};
}

\foreach \x in {25}
\foreach \y in {6} 
{ 
\draw [fill=orange] (\x+59,60+\y)--(\x+59,59+\y)--(\x+60,59+\y)--(\x+60,58+\y)--(\x+62,58+\y)--(\x+62,57+\y)--(\x+63,57+\y)--(\x+63,56+\y)--(\x+64,56+\y)--(\x+64,54+\y)--(\x+65,54+\y)--(\x+65,53+\y)--(\x+66,53+\y)--(\x+66,52+\y)--(\x+68,52+\y)--(\x+68,51+\y)--(\x+69,51+\y)--(\x+69,53+\y)--(\x+71,53+\y)--(\x+71,54+\y)--(\x+70,54+\y)--(\x+70,56+\y)--(\x+69,56+\y)--(\x+69,57+\y)--(\x+68,57+\y)--(\x+68,58+\y)--(\x+66,58+\y)--(\x+66,59+\y)--(\x+65,59+\y)--(\x+65,60+\y)--(\x+64,60+\y)--(\x+64,62+\y)--(\x+63,62+\y)--(\x+63,63+\y)
--(\x+62,63+\y)--(\x+62,64+\y)--(\x+60,64+\y)--(\x+60,65+\y)--(\x+59,65+\y)--(\x+59,66+\y)--(\x+58,66+\y)--(\x+58,68+\y)--(\x+57,68+\y)--(\x+57,69+\y)--(\x+56,69+\y)--(\x+56,70+\y)--(\x+54,70+\y)--(\x+54,71+\y)--(\x+53,71+\y)--(\x+51,71+\y)--(\x+51,68+\y)--(\x+52,68+\y)--(\x+52,66+\y)--(\x+53,66+\y)--(\x+53,65+\y)--(\x+54,65+\y)--(\x+54,64+\y)--(\x+56,64+\y)--(\x+56,63+\y)--(\x+57,63+\y)--(\x+57,62+\y)--(\x+58,62+\y)--(\x+58,60+\y)--(\x+59,60+\y);
\draw [color=black!50] (\x+57,63+\y)--(\x+57,65+\y)--(\x+59,65+\y);
\draw [color=black!50] (\x+63,57+\y)--(\x+63,59+\y)--(\x+65,59+\y);
\node at (\x+60.5,58.5+\y) {$\cdot$}; \node at (\x+61.5,58.5+\y) {$\cdot$}; \node at (\x+64.5,55.5+\y) {$\cdot$}; \node at (\x+64.5,54.5+\y) {$\cdot$};
\node at (\x+66.5,52.5+\y) {$\cdot$}; \node at (\x+67.5,52.5+\y) {$\cdot$}; \node at (\x+69.5,55.5+\y) {$\cdot$}; \node at (\x+69.5,54.5+\y) {$\cdot$};
\node at (\x+66.5,57.5+\y) {$\cdot$}; \node at (\x+67.5,57.5+\y) {$\cdot$}; \node at (\x+63.5,61.5+\y) {$\cdot$}; \node at (\x+63.5,60.5+\y) {$\cdot$};
\node at (\x+60.5,63.5+\y) {$\cdot$}; \node at (\x+61.5,63.5+\y) {$\cdot$}; \node at (\x+57.5,67.5+\y) {$\cdot$}; \node at (\x+57.5,66.5+\y) {$\cdot$};
\node at (\x+54.5,69.5+\y) {$\cdot$}; \node at (\x+55.5,69.5+\y) {$\cdot$}; \node at (\x+52.5,67.5+\y) {$\cdot$}; \node at (\x+52.5,66.5+\y) {$\cdot$}; 
\node at (\x+54.5,64.5+\y) {$\cdot$}; \node at (\x+55.5,64.5+\y) {$\cdot$}; \node at (\x+58.5,61.5+\y) {$\cdot$}; \node at (\x+58.5,60.5+\y) {$\cdot$};
}


\foreach \x in {0}
\foreach \y in {0} 
{ 
\draw [fill=lime] (\x+69,53+\y)--(\x+69,50+\y)--(\x+70,50+\y)--(\x+70,48+\y)--(\x+71,48+\y)--(\x+71,47+\y)--(\x+72,47+\y)--(\x+72,46+\y)--(\x+74,46+\y)--(\x+74,45+\y)--(\x+75,45+\y)--(\x+75,47+\y)--(\x+77,47+\y)--(\x+77,48+\y)--(\x+76,48+\y)--(\x+76,50+\y)--(\x+75,50+\y)--(\x+75,51+\y)--(\x+74,51+\y)--(\x+74,52+\y)--(\x+72,52+\y)--(\x+72,53+\y)--(\x+69,53+\y);
}

\foreach \x in {-24}
\foreach \y in {24} 
{ 
\draw [fill=lime] (\x+69,53+\y)--(\x+69,50+\y)--(\x+70,50+\y)--(\x+70,48+\y)--(\x+71,48+\y)--(\x+71,47+\y)--(\x+72,47+\y)--(\x+72,46+\y)--(\x+74,46+\y)--(\x+74,45+\y)--(\x+75,45+\y)--(\x+75,47+\y)--(\x+77,47+\y)--(\x+77,48+\y)--(\x+76,48+\y)--(\x+76,50+\y)--(\x+75,50+\y)--(\x+75,51+\y)--(\x+74,51+\y)--(\x+74,52+\y)--(\x+72,52+\y)--(\x+72,53+\y)--(\x+69,53+\y);
}

\foreach \x in {7}
\foreach \y in {-7} 
{ 
\draw [fill=lime] (\x+69,53+\y)--(\x+69,50+\y)--(\x+70,50+\y)--(\x+70,48+\y)--(\x+71,48+\y)--(\x+71,47+\y)--(\x+72,47+\y)--(\x+72,46+\y)--(\x+74,46+\y)--(\x+74,45+\y)--(\x+75,45+\y)--(\x+75,47+\y)--(\x+77,47+\y)--(\x+77,48+\y)--(\x+76,48+\y)--(\x+76,50+\y)--(\x+75,50+\y)--(\x+75,51+\y)--(\x+74,51+\y)--(\x+74,52+\y)--(\x+72,52+\y)--(\x+72,53+\y)--(\x+69,53+\y);
}
\foreach \x in {13}
\foreach \y in {-13} 
{ 
\draw [fill=lime] (\x+69,53+\y)--(\x+69,50+\y)--(\x+70,50+\y)--(\x+70,48+\y)--(\x+71,48+\y)--(\x+71,47+\y)--(\x+72,47+\y)--(\x+72,46+\y)--(\x+74,46+\y)--(\x+74,45+\y)--(\x+75,45+\y)--(\x+75,47+\y)--(\x+77,47+\y)--(\x+77,48+\y)--(\x+76,48+\y)--(\x+76,50+\y)--(\x+75,50+\y)--(\x+75,51+\y)--(\x+74,51+\y)--(\x+74,52+\y)--(\x+72,52+\y)--(\x+72,53+\y)--(\x+69,53+\y);
}

\foreach \x in {25}
\foreach \y in {6} 
{ 
\draw [fill=lime] (\x+69,53+\y)--(\x+69,50+\y)--(\x+70,50+\y)--(\x+70,48+\y)--(\x+71,48+\y)--(\x+71,47+\y)--(\x+72,47+\y)--(\x+72,46+\y)--(\x+74,46+\y)--(\x+74,45+\y)--(\x+75,45+\y)--(\x+75,47+\y)--(\x+77,47+\y)--(\x+77,48+\y)--(\x+76,48+\y)--(\x+76,50+\y)--(\x+75,50+\y)--(\x+75,51+\y)--(\x+74,51+\y)--(\x+74,52+\y)--(\x+72,52+\y)--(\x+72,53+\y)--(\x+69,53+\y);
}
\foreach \x in {31}
\foreach \y in {0} 
{ 
\draw [fill=lime] (\x+69,53+\y)--(\x+69,50+\y)--(\x+70,50+\y)--(\x+70,48+\y)--(\x+71,48+\y)--(\x+71,47+\y)--(\x+72,47+\y)--(\x+72,46+\y)--(\x+74,46+\y)--(\x+74,45+\y)--(\x+75,45+\y)--(\x+75,47+\y)--(\x+77,47+\y)--(\x+77,48+\y)--(\x+76,48+\y)--(\x+76,50+\y)--(\x+75,50+\y)--(\x+75,51+\y)--(\x+74,51+\y)--(\x+74,52+\y)--(\x+72,52+\y)--(\x+72,53+\y)--(\x+69,53+\y);
}

\node at (61,61) {\textbf{2}}; \node at (92,61) {\textbf{3}};  \node at (92,30) {\textbf{1}};

\end{tikzpicture}
\end{center}
\caption{The tiling pattern of a floor.}\label{fig_pattern}
\end{figure}
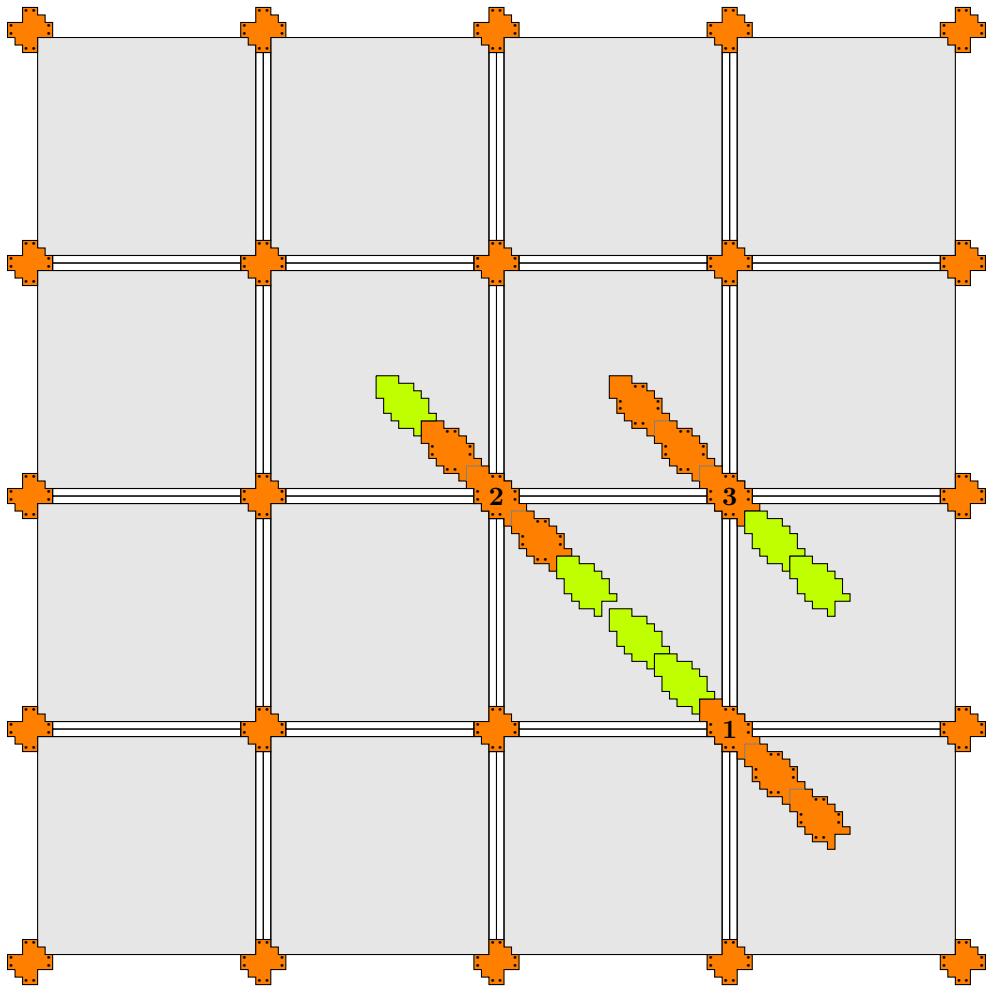

    \item Because building blocks $A$ and $B$ only appear inside the mouths of the jaws, a meat polycube must be gripped by two jaws from two sides (northwest and southeast directions), leaving a single simulated Wang tile outside. There is a degree of freedom to choose which of the $3$ simulated Wang tiles to be left out (see Figure \ref{fig_pattern}). The building blocks $O$ outside the jaws enforce another two jaws to be placed in the northeast and southwest directions. By repeating the above arguments and extending the partial tiling outwards, we get the pattern illustrated in Figure \ref{fig_pattern}. As we have mentioned in the previous paragraph, if any one of the orange meat is replaced by $3$ fillers, then there is no way to fill the gaps outside the jaws and it fails to tile the entire space.

    \item The tiling pattern in Figure \ref{fig_pattern} that extends horizontally and infinitely is called a \textit{floor} of the space tiling. A floor consists of $8$ unit layers (ignoring the bumps of building blocks $\mathbb{M}$, $\mathbb{J}$ and $\mathbb{F}$) as most of the building blocks of the meats, jaws and fillers are $8\times 8\times 8$ polycubes. The small solid dots on the orange meats in Figure \ref{fig_pattern} denote the building blocks $c$. Most of the fillers and meats inside the jaws are omitted in Figure \ref{fig_pattern}. The building blocks $\mathbb{M}$, $\mathbb{J}$ and $\mathbb{F}$ enforce the tiling must form a vertically two-way infinite stack of identical floors, where each floor is in the pattern illustrated in Figure \ref{fig_pattern}.

\item Finally, there are some leftover gaps by the stacks of floors. The small gaps inside the jaws can be always filled by tooth polycubes. The gaps outside the jaws can be filled by link polycubes without gaps or overlaps if and only if the building blocks $c$ connected by the links are at the same altitude. This is equivalent to that the colors encoded by the building block $c$ must be the same for any pair of adjacent sides of the simulated Wang tiles outside the jaws.

\end{itemize}

So we have reduced each instance of Wang's domino problem (i.e. a set $W$ of Wang tiles) to an instance of the translational tiling problem of the space (i.e. a set $P$ of $6$ polycubes). As Wang's domino problem is undecidable (Thoerem \ref{thm_berger}), the translational tiling problem with a set of $6$ polycubes is also undecidable. \end{proof}

\noindent \textbf{Remark 1}. Note that in the proof of Theorem \ref{thm_main}, the link polycube may or may not be aligned (regarding the altitude) with the floors. If a link polycube connects two building blocks $c$ attached to the lower part of the meats, then it is aligned with the floor (see the top of Figure \ref{fig_misalign}, the link lies in the same $8$ unit layers in space with the meats that it connects). If a link polycube connects two building blocks $c$ attached to the upper part of the meats, then it is misaligned with the floors (see the bottom of Figure \ref{fig_misalign}). Allowing misalignment is crucial in proving our main result, it helps to decrease the number of tooth polycubes to just one. Therefore, the techniques we employ to show the undecidability of translational tiling with just $6$ tiles make use of the $3$-dimensional space, which may not be applicable to the translational tiling of the plane.


\begin{figure}[H]
\begin{center}
\begin{tikzpicture}[scale=0.5]

\draw [fill=gray!20] (1,0)--(26,0)--(26,1)--(1,1)--(1,0);

\draw [fill=gray!20] (1,-2)--(26,-2)--(26,-1)--(1,-1)--(1,-2);
\draw [fill=violet!30] (27,-1.5)--(26,-1.5)--(26,-1)--(27,-1)--(27,-1.5);

\draw [fill=violet!30] (1,-1.5)--(0,-1.5)--(0,-1)--(1,-1)--(1,-1.5);
 
\draw [fill=orange] (-2,0)--(1,0)--(1,0.5)--(0, 0.5)--(0,1)--(-2,1);
\draw [fill=orange] (29,0)--(26,0)--(26,0.5)--(27, 0.5)--(27,1)--(29,1);

\draw [fill=orange] (-2,-2.5)--(0,-2.5)--(0,-2)--(1, -2)--(1,-1.5)--(-2,-1.5);
\draw [fill=orange] (29,-2.5)--(27,-2.5)--(27,-2)--(26, -2)--(26,-1.5)--(29,-1.5);

\foreach \x in{2,...,25}
{
\draw [color=black!60] (\x,0)--(\x,1);
\draw [color=black!60] (\x,-2)--(\x,-1);
}

\draw [fill=violet!30] (1,1)--(0,1)--(0,0.5)--(1,0.5)--(1,1);\draw [fill=violet!30] (27,1)--(26,1)--(26,0.5)--(27,0.5)--(27,1);

\end{tikzpicture}
\end{center}
\caption{The links may be misaligned with the floors (side view).}\label{fig_misalign}
\end{figure}
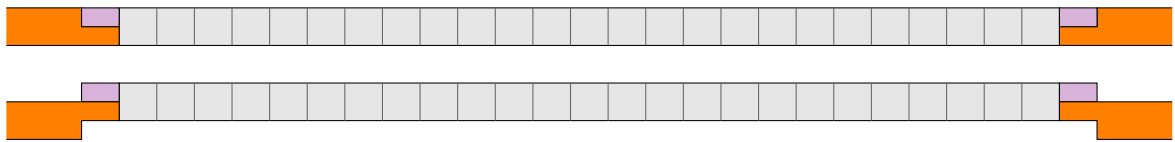

\noindent \textbf{Remark 2}. As a consequence of Theorem \ref{thm_main}, translational tiling of $\mathbb{Z}^n (n\geq 4)$ with a set of $6$ tiles is undecidable.

\section{Conclusion}\label{sec_conclude}

In this paper, we show that translational tiling of $\mathbb{Z}^3$ is undecidable with a set of $6$ tiles, even if all the tiles are corresponding to connected polycubes. For the undecidability of general translational tiling of $\mathbb{Z}^n$ with a set of $k$ tiles, the following problems are interesting for future study.

\begin{Problem}
    Is it undecidable for translational tiling of $\mathbb{Z}^3$ or $\mathbb{Z}^4$ with a set of $5$ tiles?
\end{Problem}

\begin{Problem}[\cite{gt24b}]
    Is there a fixed $n$ such that translational tilings of $\mathbb{Z}^n$ with a single tile is undecidable?
\end{Problem}

\section*{Acknowledgements}
The first author was supported by the Research Fund of Guangdong University of Foreign Studies (Nos. 297-ZW200011 and 297-ZW230018), and the National Natural Science Foundation of China (No. 61976104).



\newpage
\section*{Appendix}

Supplementary figures are given in the appendix. Figure \ref{fig_segment} illustrates a filler or a segment of the meat polycube in which $4$ consecutive building blocks $c$ are used to encode colors. The building blocks $\mathbb{F}$, $\mathbb{M}$ and $c$ may be replaced by normal $8\times 8\times 8$ functional cubes depending on whether it is a filler or a segment of the meat. More details of a meat and four surrounding jaws are illustrated in Figure \ref{fig_meat_in_jaw}, where the meat is shown in orange lines.


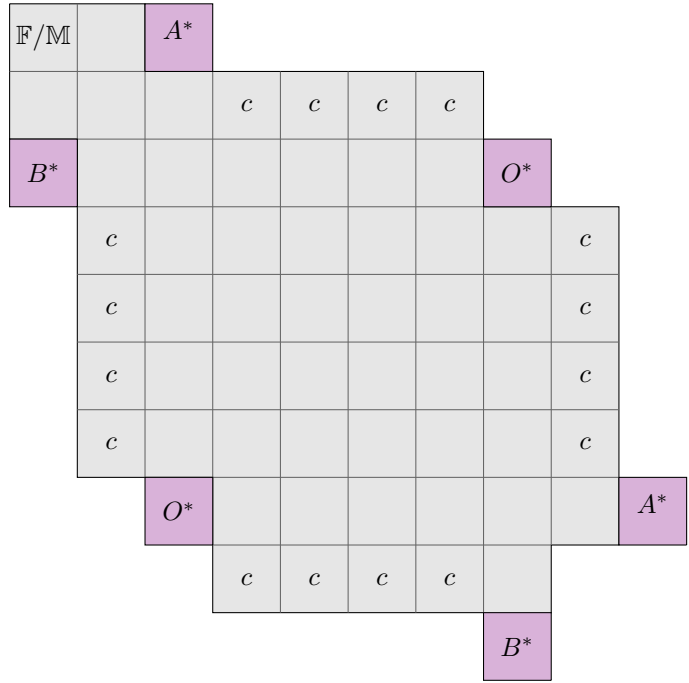
\begin{figure}[H]
\begin{center}
\begin{tikzpicture}[scale=0.1]

\draw [fill=gray!20] (0,63)--(18,63)
--(18,54)
--(63,54)--(63,36)--(81,36)
--(81,0)
--(81,-9)--(72,-9)--(72,-18)
--(27,-18)
--(27,0)
--(9,0)--(9,36)
--(9,45)--(0,45)--(0,63);

\draw [color=black!60] (0,54)--(18,54)--(18,0);
\draw [color=black!60] (9,63)--(9,45)--(63,45);
\draw [color=black!60] (9,36)--(63,36)--(63,-18);
\draw [color=black!60] (9,27)--(81,27);
\draw [color=black!60] (9,18)--(81,18);
\draw [color=black!60] (9,9)--(81,9);

\draw [color=black!60] (27,0)--(81,0);
\draw [color=black!60] (27,-9)--(72,-9);

\draw [color=black!60] (27,0)--(27,54);
\draw [color=black!60] (36,-18)--(36,54);
\draw [color=black!60] (45,-18)--(45,54);
\draw [color=black!60] (54,-18)--(54,54);
\draw [color=black!60] (72,-9)--(72,36);

\foreach \x in {0,63}
{
\draw [fill=violet!30] (\x+27,54-\x)--(\x+27,63-\x)--(\x+18,63-\x)--(\x+18,54-\x)--(\x+27,54-\x);
\draw [fill=violet!30] (\x+9,36-\x)--(\x+9,45-\x)--(\x+0,45-\x)--(\x+0,36-\x)--(\x+9,36-\x);
\node at (\x+22.5,59.5-\x) {$A^*$};
\node at (\x+4.5,40.5-\x) {$B^*$};
}

\foreach \x in {18}
{
\draw [fill=violet!30] (\x+54,54-\x)--(\x+54,63-\x)--(\x+45,63-\x)--(\x+45,54-\x)--(\x+54,54-\x);
\node at (\x+49.5,58.5-\x) {$O^*$};
}

\foreach \x in {0}
{
\draw [fill=violet!30] (\x+27,-9-\x)--(\x+27,0-\x)--(\x+18,0-\x)--(\x+18,-9-\x)--(\x+27,-9-\x);
\node at (\x+22.5,-4.5-\x) {$O^*$};
}

\node at (4.5,58.5) {$\mathbb{F}$/$\mathbb{M}$};

\foreach \x in {27,36,45,54}
{
\node at (\x+4.5,49.5) {$c$};
\node at (\x+4.5,-13.5) {$c$};
}

\foreach \x in {0,9,18,27}
{
\node at (13.5,4.5+\x) {$c$};
\node at (76.5,4.5+\x) {$c$};
}

\end{tikzpicture}
\end{center}
\caption{A filler or a segment of a meat.}\label{fig_segment}
\end{figure}


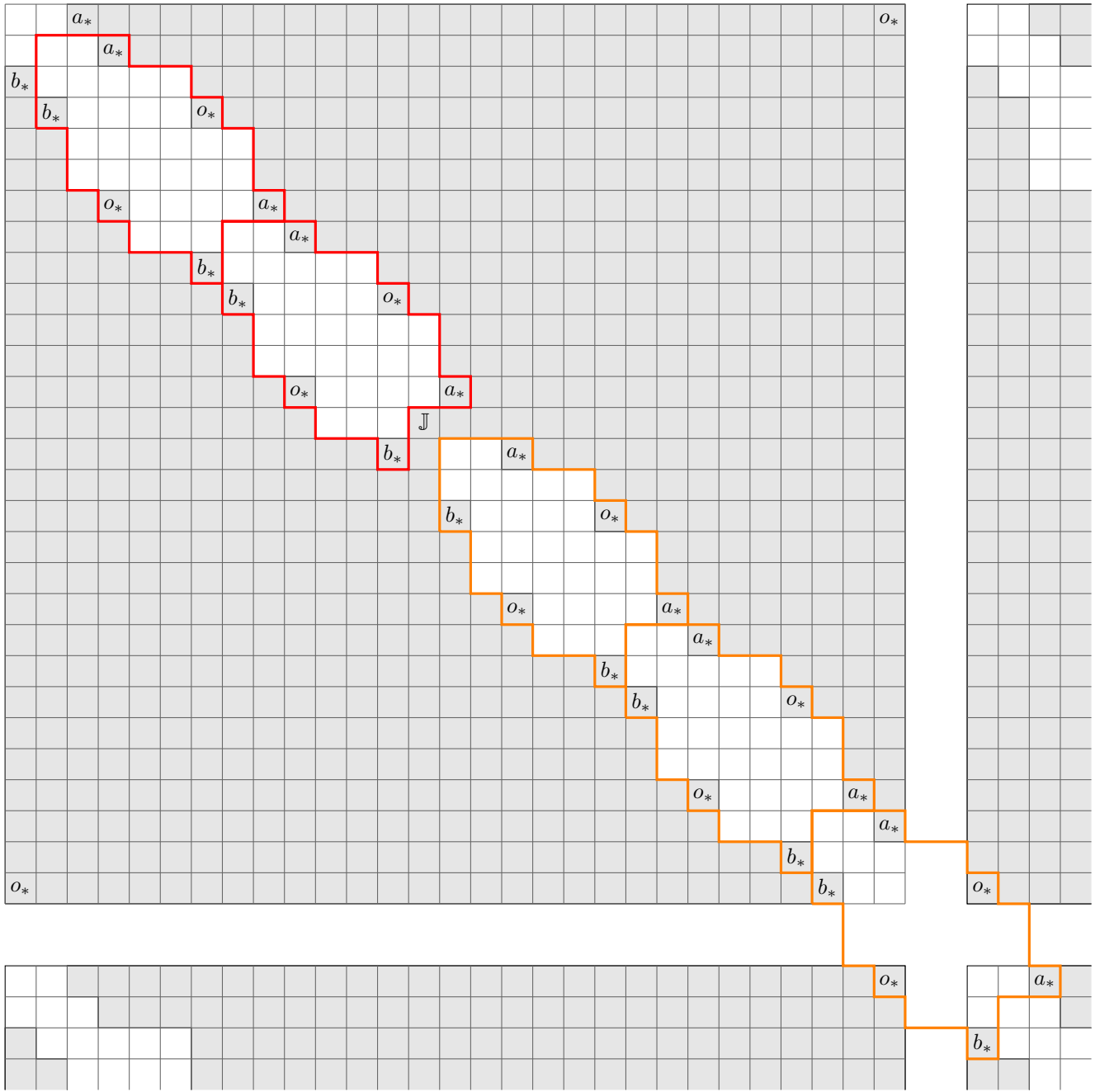
\begin{figure}[H]
\begin{center}
\begin{tikzpicture}[scale=0.5]

\draw [fill=gray!20] (0,0)--(27,0)
--(27,1)--(26,1)--(26,2)--(23,2)--(23,4)--(21,4)--(21,7)--(20,7)--(20,8)--(17,8)--(17,10)--(15,10)--(15,13)--(14,13)--(14,15)--(16,15)--(16,14)--(19,14)--(19,12)--(21,12)--(21,9)--(22,9)--(22,8)--(25,8)--(25,6)--(27,6)--(27,3)--(28,3)--(28,2)
--(29,2)
--(29,29)
--(2,29)
--(2,28)--(3,28)--(3,27)--(6,27)--(6,25)--(8,25)--(8,22)--(9,22)--(9,21)--(12,21)--(12,19)--(14,19)--(14,16)--(13,16)--(13,15)--(10,15)--(10,17)--(8,17)--(8,20)--(7,20)--(7,21)--(4,21)--(4,23)--(2,23)--(2,26)--(1,26)--(1,27)
--(0,27)--(0,0);

\foreach \x in {0,...,29}  
{ 
\draw [ color=black!60] (\x,0)--(\x,29);
\draw [ color=black!60] (0,\x)--(29,\x);
}

\foreach \x in {0,1,6,7,12,14,19,20,25,26}
{
\node at (\x+2.5,28.5-\x) {$a_*$};
\node at (\x+0.5,26.5-\x) {$b_*$};
}

\foreach \x in {0,6,13,19}
{
\node at (\x+6.5,25.5-\x) {$o_*$};
\node at (\x+3.5,22.5-\x) {$o_*$};
}
\node at (0.5,0.5) {$o_*$}; \node at (28.5,28.5) {$o_*$};

\node at (13.5,15.5) {$\mathbb{J}$};


\draw  (31,-6)--(31,-2)--(35,-2);

\draw (0,-6)--(0,-2)--(29,-2)--(29,-6);
\draw (35,0)--(31,0)--(31,29)--(35,29);

\draw [fill=gray!20] (29,-6)--(29,-2)--(2,-2)--(2,-3)--(3,-3)--(3,-4)--(6,-4)--(6,-6);

\draw [fill=gray!20] (0,-6)--(0,-4)--(1,-4)--(1,-5)--(2,-5)--(2,-6);
\draw [fill=gray!20] (31+0,-6)--(31+0,-4)--(31+1,-4)--(31+1,-5)--(31+2,-5)--(31+2,-6);
\draw [fill=gray!20] (35,29)--(33,29)--(33,28)--(34,28)--(34,27)--(35,27);
\draw [fill=gray!20] (35,29-31)--(33,29-31)--(33,28-31)--(34,28-31)--(34,27-31)--(35,27-31);

\draw [fill=gray!20] (35,0)--(31,0)--(31,27)--(32,27)--(32,26)--(33,26)--(33,23)--(35,23);

\foreach \x in {-3,-4,-5}  
{ 
\draw [ color=black!60] (0,\x)--(29,\x);\draw [ color=black!60] (31,\x)--(35,\x);
}
\foreach \x in {32,33,34}  
{ 
\draw [ color=black!60] (\x,0)--(\x,29);
\draw [ color=black!60] (\x,-2)--(\x,-6);

}
\foreach \x in {1,...,28}
{
\draw [ color=black!60] (\x,-2)--(\x,-6);
\draw [ color=black!60] (31,\x)--(35,\x);
}

\draw [orange, very thick] (14,12)--(14,15)--(17,15)--(17,14)--(19,14)--(19,13)--(20,13)--(20,12)--(21,12)--(21,10)--(22,10)--(22,9)--(20,9)--(20,7)--(19,7)--(19,8)--(17,8)--(17,9)--(16,9)--(16,10)--(15,10)--(15,12)--(14,12);

\foreach \x in{6}
\foreach \y in{-6}
{
\draw [orange, very thick](\x+14,12+\y)--(\x+14,15+\y)--(\x+17,15+\y)--(\x+17,14+\y)--(\x+19,14+\y)--(\x+19,13+\y)--(\x+20,13+\y)--(\x+20,12+\y)--(\x+21,12+\y)--(\x+21,10+\y)--(\x+22,10+\y)--(\x+22,9+\y)--(\x+20,9+\y)--(\x+20,7+\y)--(\x+19,7+\y)--(\x+19,8+\y)--(\x+17,8+\y)--(\x+17,9+\y)--(\x+16,9+\y)--(\x+16,10+\y)--(\x+15,10+\y)--(\x+15,12+\y)--(\x+14,12+\y);
}

\foreach \x in{12}
\foreach \y in{-12}
{
\draw [orange, very thick](\x+14,12+\y)--(\x+14,15+\y)--(\x+17,15+\y)--(\x+17,14+\y)--(\x+19,14+\y)--(\x+19,13+\y)--(\x+20,13+\y)--(\x+20,12+\y)--(\x+21,12+\y)--(\x+21,10+\y)--(\x+22,10+\y)--(\x+22,9+\y)--(\x+20,9+\y)--(\x+20,7+\y)--(\x+19,7+\y)--(\x+19,8+\y)--(\x+17,8+\y)--(\x+17,9+\y)--(\x+16,9+\y)--(\x+16,10+\y)--(\x+15,10+\y)--(\x+15,12+\y)--(\x+14,12+\y);
}

\foreach \x in{-13}
\foreach \y in{13}
{
\draw [red, very thick](\x+14,12+\y)--(\x+14,15+\y)--(\x+17,15+\y)--(\x+17,14+\y)--(\x+19,14+\y)--(\x+19,13+\y)--(\x+20,13+\y)--(\x+20,12+\y)--(\x+21,12+\y)--(\x+21,10+\y)--(\x+22,10+\y)--(\x+22,9+\y)--(\x+20,9+\y)--(\x+20,7+\y)--(\x+19,7+\y)--(\x+19,8+\y)--(\x+17,8+\y)--(\x+17,9+\y)--(\x+16,9+\y)--(\x+16,10+\y)--(\x+15,10+\y)--(\x+15,12+\y)--(\x+14,12+\y);
}
\foreach \x in{-7}
\foreach \y in{7}
{
\draw [red, very thick](\x+14,12+\y)--(\x+14,15+\y)--(\x+17,15+\y)--(\x+17,14+\y)--(\x+19,14+\y)--(\x+19,13+\y)--(\x+20,13+\y)--(\x+20,12+\y)--(\x+21,12+\y)--(\x+21,10+\y)--(\x+22,10+\y)--(\x+22,9+\y)--(\x+20,9+\y)--(\x+20,7+\y)--(\x+19,7+\y)--(\x+19,8+\y)--(\x+17,8+\y)--(\x+17,9+\y)--(\x+16,9+\y)--(\x+16,10+\y)--(\x+15,10+\y)--(\x+15,12+\y)--(\x+14,12+\y);
}

\foreach \x in {31}
{
\node at (\x+2.5,28.5-\x) {$a_*$};
\node at (\x+0.5,26.5-\x) {$b_*$};
}

\foreach \x in {25}
{
\node at (\x+6.5,25.5-\x) {$o_*$};
\node at (\x+3.5,22.5-\x) {$o_*$};
}
 
\end{tikzpicture}
\end{center}
\caption{A meat and its surrounding jaws.}\label{fig_meat_in_jaw}
\end{figure}


\begin{thebibliography}{99}

\bibitem{ags92}
R. Ammann, B. Gr\"unbaum, G. C. Shephard, Aperiodic tiles, \textit{Discrete \& Computational Geometry}, \textbf{8}(1992), 1-25.

\bibitem{bg13}
M. Baake, U. Grimm, Aperiodic Order, Volume 1: A Mathematical Invitation, Cambridge University Press, 2013.

\bibitem{bn91}
D. Beauquier, M. Nivat, On translating one polyomino to tile the plane, \textit{Discrete \& Computational Geometry}, \textbf{6}(1991), 575-592.

\bibitem{b66} R. Berger, The undecidability of the domino problem, \textit{Memoirs of the American Mathematical Society}, \textbf{66}(1966), 1-72.

\bibitem{b20}
B. Bhattacharya. Periodicity and decidability of tilings of $\mathbb{Z}^2$. \textit{American Journal of Mathematics}, \textbf{142}(2020), 255-266.

\bibitem{c96} K. Culik II, An aperiodic set of 13 wang tiles, \textit{Discrete Mathematics}, \textbf{160}(1996), 245-251.


\bibitem{g70}
S. W. Golomb, Tiling with a set of polyominoes, \textit{Journal of Combinatorial Theory}, \textbf{9}(1970), 60-71.

\bibitem{gt21}
R. Greenfeld, T. Tao. The structure of translational tilings in $\mathbb{Z}^d$. Discrete Analysis. (2021:16). 1-28.

\bibitem{gt23}
R. Greenfeld, T. Tao, Undecidable translational tilings with only two tiles, or one nonabelian tile. \textit{Discrete \& Computational Geometry}, \textbf{70}(2023), 1652–1706.

\bibitem{gt24a}
R. Greenfeld, T. Tao, A counterexample to the periodic tiling conjecture. \textit{Annals of Mathematics}, \textbf{200}(1)(2024), 301-363.

\bibitem{gt24b}
R. Greenfeld, T. Tao, Undecidability of translational monotilings. to appear in \textit{Journal of the European Mathematical Society}, arXiv:2309.09504 [math.CO]

\bibitem{gs16}
B. Gr\"unbaum, G. C. Shephard, Tilings and Patterns, 2nd Edition, Dover Publications, 2016.

\bibitem{hl11}
W.-G. Hu, S.-S. Lin, Nonemptiness problems of plane square tiling with two colors. \textit{Proceedings of the American Mathematical Society}, \textbf{139}(2011), 1045–1059.

\bibitem{jr21} E. Jeandel and M. Rao, An aperiodic set of 11 Wang tiles, \textit{Advances in Combinatorics}, (2021:1), 1-37.

\bibitem{kari96} J. Kari, A small aperiodic set of wang tiles, \textit{Discrete Mathematics}, \textbf{160}(1996) 259-264.

\bibitem{kolo03}
M. N. Kolountzakis, Translational tilings of the integers with long periods. \textit{The Electronic Journal of Combinatorics}, \textbf{10}(2003), R22:1-9.

\bibitem{o09}
N. Ollinger, Tiling the plane with a fixed number of polyominoes, In: A.H. Dediu, A.M. Ionescu, C. Martín-Vide (eds), Language and Automata Theory and Applications (LATA 2009). Lecture Notes in Computer Science, vol 5457. Springer, Berlin, Heidelberg, 638-649.

\bibitem{p79}
R. Penrose, Pentaplexity a class of non-periodic tilings of the plane, \textit{The Mathematical Intelligencer}, \textbf{2}(1979), 32–37.

\bibitem{r71} R. Robinson, Undecidability and nonperiodicity for tilings of the plane, \textit{Inventiones Mathematicae}, \textbf{12}(1971), 177--209.

\bibitem{s93}
V. Sidorenko, Periodicity of one-dimensional tilings. In: A. Chmora, S.B. Wicker (eds), Error Control, Cryptology, and Speech Compression (ECCSP 1993). Lecture Notes in Computer Science, vol 829. Springer, Berlin, Heidelberg.  103-108.

\bibitem{smith23a} D. Smith, J. S. Myers, C. S. Kaplan, C. Goodman-Strauss, An aperiodic monotile. \textit{Combinatorial Theory}. \textbf{4}(1)(2024), \#6. arXiv:2303.10798 [math.CO]

\bibitem{smith23b} D. Smith, J. S. Myers, C. S. Kaplan, C. Goodman-Strauss, A chiral aperiodic monotile, arXiv:2305.17743 [math.CO]

\bibitem{st05}
J. P. Steinberger, Indecomposable tilings of the integers with exponentially long periods. \textit{The Electronic Journal of Combinatorics}, \textbf{12}(2005), R36:1-20.

\bibitem{wang61}
H. Wang, Proving theorems by pattern recognition-II, \textit{Bell System Technical Journal}, \textbf{40}(1961) 1-41.


\bibitem{w15} A. Winslow, An optimal algorithm for tiling the plane with a translated polyomino,
In: K. Elbassioni, K. Makino (eds), Algorithms and Computation (2015), Springer, Berlin, Heidelberg, 3-13.

\bibitem{yang23} C. Yang, Tiling the plane with a set of ten polyominoes, \textit{International Journal of Computational Geometry \& Applications}, \textbf{33}(03n04)(2023), 55-64.

\bibitem{yang23b} C. Yang, On the undecidability of tiling the plane with a set of $9$ polyominoes (in Chinese), (2024), to appear in \textit{SCIENTIA SINICA Mathematica}. \url{https://doi.org/10.1360/SSM-2024-0035}


\bibitem{yz24} 
C. Yang, Z. Zhang, A proof of Ollinger's conjecture: undecidability of tiling the plane with a set of 8 polyominoes, arXiv:2403.13472 [math.CO]
 
\end{thebibliography}
\end{document}